\documentclass[11pt,leqno]{article}
\usepackage{amsmath,amsthm}
\usepackage{amsfonts}
\usepackage{mathrsfs}
\usepackage{hyperref}
\usepackage{textcomp}
\usepackage{amssymb}
\usepackage{epsfig}
\usepackage{graphicx}
\usepackage{caption}
\usepackage{bm}
\usepackage{color} 

\numberwithin{equation}{section} 
\newtheorem{theorem}{\bf Theorem}[section]
\newtheorem{example}{\bf Example}[section]

\newtheorem{remark}{\bf Remark}[section]
\newtheorem{lemma}{\bf Lemma}[section]

\newcommand{\bphi}{\mbox{\boldmath $\phi$}}

\newcommand{\bF}{{\bf {F}}}

 \newcommand{\norm}[1]{\left\lVert #1\right\rVert}
\newcommand{\tnorm}[1]{\ensuremath{\left| \! \left| \! \left|} #1 \ensuremath{\right| \! \right| \! \right|}}
\newsavebox{\savepar}		 

\begin{document}
\title{Global Stabilization of 2D Forced Viscous Burgers' Equation Around Nonconstant Steady State Solution by Nonlinear Neumann Boundary Feedback Control:Theory and Finite Element Analysis}
\author{ Sudeep Kundu\footnote{
		Institute of Mathematics and Scientific Computing, University of Graz,
		Heinrichstr. 36, A-8010 Graz, Austria,
		Email:sudeep.kundu@uni-graz.at}\quad 
and
	Amiya Kumar Pani\footnote{
		Department of Mathematics, 	IIT Bombay, 
		Powai, Mumbai-400076, India,
		Email:{akp@math.iitb.ac.in}}
	.}
\maketitle
\abstract{Global stabilization of viscous Burgers' equation around constant steady state solution has been discussed in the literature. The main objective of this paper is to show global stabilization results for the 2D forced viscous Burgers' equation around a nonconstant steady state solution using nonlinear Neumann boundary feedback control law, under some smallness condition on that steady state solution.
On discretizing in space using $C^0$ piecewise linear elements keeping time variable
continuous, a semidiscrete scheme is obtained. Moreover, global stabilization results for the semidiscrete solution and optimal error estimates for the state variable in $L^\infty(L^2)$ and $L^\infty(H^1)$-norms  are derived. Further, optimal convergence result is established for the boundary feedback control law. All our results in this paper preserve exponential stabilization property. 
Finally, some numerical experiments are documented to confirm our theoretical findings.}

\section{Introduction}
Consider the following  two-dimensional forced viscous Burgers' equation 
with Neumann boundary control
: seek $u=u(x,t),$ $t>0$ which satisfies
\begin{align}
&u_t-\nu\Delta u+u(\nabla u\cdot {\bf{1}})=f^\infty\qquad\text{in}\quad (x,t)\in \Omega\times(0,\infty)\label{feq1.1},\\
&\frac{\partial u}{\partial n}(x,t)=v_2(x,t)\qquad \text{on} \quad(x,t)\in \partial \Omega \times (0,\infty)\label{feq1.2},\\
&u(x,0)=u_0(x)\qquad x\in \Omega\label{feq1.3},
\end{align}
where $\Omega$ is a bounded subset in $\mathbb{R}^2$ with smooth boundary, $\nu>0$ is a constant, $v_2$ is a scalar control input, ${\bf{1}}=(1,1)$, $\nu\Delta u$ is the diffusive term, $u(\nabla u\cdot {\bf{1}})=u\sum_{i=1}^{2}u_{x_i}$ is the nonlinear convective term, forcing function  $f^\infty=f^\infty(x)$ is independent of the time and 
$u_0$ is a given function. When $\nu=0$, it is called inviscid Burgers' equation, which is studied in nonlinear wave propagation.\\
For one dimensional Burgers' equation, local stabilization results  are documented in \cite{bk}, \cite {bk1}
for distributed and Dirichlet boundary control  and in \cite {bgs}  for Neumann boundary control  with small initial
 data. For more details, see \cite {ik}, 
 \cite{iy} and \cite{lmt}.
Using linear feedback control law for the linearized part as in Navier-Stokes equations  \cite{Raymond06},  local stabilizability results can be proved for the two dimensional Burgers' equation. In \cite{raymond2010}, the authors have shown local stabilization results for the two dimensional Burgers' equation directly through a nonlinear feedback control law and several numerical experiments are also reported in their article conforming their theoretical results. 
  Subsequently
 Buchot {\it {et al.}} \cite{raymond2015} have derived local stabilization results in the case of partial information for the two dimensional Burgers' type equation. \\
Based on Lyapunov type functional, global stabilization result around constant steady state solution for one dimensional Burgers' equation is derived in \cite{krstic1} and \cite {balogh}
for both Dirichlet and Neumann boundary control laws. For more detailed references, we 
refer to \cite{liu}, \cite{Smaoui} and \cite{smaoui1}. In \cite{camphouse04}, authors implement Dirichlet boundary feedback control law on the obstacles 
for two dimensional Burgers' equation by solving both Riccati equation and Chandrasekhar equations.
 In \cite{skakp1}, it is shown that using a nonlinear Neumann boundary feedback 
 control laws, solution of 1D Burgers' equation converges exponentially to its constant steady state solution in $L^\infty(L^2)$, $L^\infty(L^\infty)$, $L^\infty(H^1)$ and $L^\infty(H^2)$ norms. Then, an application of finite element method in 
 spatial direction yields a semidiscrete system and global stabilization results are proved for the semidiscrete approximation. Finally, optimal error
 estimates for the state variable and superconvergence result for the control laws are established. This analysis is then extended to 2D
 Burgers' equation in \cite{skakp2}, and stabilization results depicting convergence of the solution to its constant steady state solution are derived. Moreover, convergence  result for the nonlinear feedback 
 control law is also documented.\\
  But when the steady state solution is non constant in the case of forced viscous Burgers' equation, it is not known whether global stabilization results still holds or not. Also, when applying finite element method, no result is available in the literature on rate of convergence. Hence, in this paper, an attempt has been made to fill this gap.\\
Now the corresponding equilibrium or steady state problem becomes: find $u^\infty$ as a solution of
\begin{align}
-\nu \Delta u^\infty+u^\infty(\nabla u^\infty\cdot {\bf{1}})&=f^\infty \qquad\text{in} \quad \Omega \label{feq1.5},\\
\frac{\partial u^\infty}{\partial n}&=0 \quad \text{on} \quad \partial \Omega \label{feq1.6}.
\end{align}
To achieve $$\lim_{t\to\infty} u(x,t)=u^\infty\quad \text{a. e.} ~ x\in\Omega,$$ it is enough to consider $\lim_{t\to\infty}w=0,$ where $w=u-u^\infty$ and $w$ satisfies
\begin{align}
&w_t-\nu \Delta w+u^\infty(\nabla w\cdot {\bf{1}})+w(\nabla u^\infty\cdot {\bf{1}})+w(\nabla w\cdot {\bf{1}})=0 \qquad\text{in}\quad (x,t)\in \Omega\times(0,\infty)\label{feq1.7},\\
&\frac{\partial w}{\partial n}(.,t)=v_2(x,t),\quad \text{on} \quad \partial \Omega\times(0,\infty)\label{eqn1.8},\\
&w(0)=u_0-u^\infty=w_0(\text{say})\quad\text{in}\quad\Omega \label{feq1.9}.
\end{align}
The motivation  to choose Neumann boundary control comes from the physical situation.
For instance, in thermal problem, one can not actuate the temperature $w$, but the heat flux $\frac{\partial w}{\partial n}$. 
Here, the control variable $v_2$ is to be chosen as a function of $w$ appropriately. 
%
In this article, we first prove global stabilization results in $L^\infty(L^2)$, $L^\infty(H^1)$- norms for the 
problem \eqref{feq1.7}-\eqref{feq1.9} with control law \eqref{feqx1}. Then, $C^0$-
conforming finite element method is applied to discretize the spatial variable, keeping
time variable continuous and global stabilization results are proved for the 
semidiscrete solution. Further, optimal error estimates are derived for the 
state variable and superconvergence results are obtained for the nonlinear Neumann boundary feedback control law.\\
The rest of the paper is organized as follows. Section $2$ contains notations, preliminary results and properties of the steady state solution. In Section $3,$ we focus on global stabilization results using nonlinear feedback control law. Section $4$ deals with finite element approximation and global stabilization results for the semidiscrete system. Further, optimal error estimates are obtained for the state variable and convergence result is derived for the feedback control law. Finally, Section $4$ concludes with some numerical experiments.
\section{Notations and properties of steady state solution}
In this section, we discuss some properties for the steady state problem \eqref{feq1.5}-\eqref{feq1.6} along with some preliminary results.\\
 Throughout this paper, we use
standard Sobolev space $H^m(\Omega)=W^{m,2}(\Omega)$ with norm $\norm{\cdot}_m,$ and seminorm $|\cdot|_m$. For $m=0,$ it corresponds to the usual $L^2$ norm and is denoted by $\norm{\cdot}$. Moreover, $(\cdot,\cdot)$ and $\langle\cdot,\cdot\rangle$ denote the innerproducts in $L^2(\Omega)$ and $L^2(\partial\Omega)$ respectively. Also we use the Sobolev space $H^s(\partial\Omega)$ of order $s$.
The space $L^p((0,T);X)$  $1\leq p\leq\infty,$ consists of all strongly measurable functions $v:[0,T] \rightarrow X $ with norm
$$\norm{v}_{L^p((0,T);X)}:=\left(\int_{0}^{T}\norm{v(t)}^p_X dt\right)^\frac{1}{p}<\infty \quad \text{for} \quad 1\leq p<\infty,$$ and 
$$\norm{v}_{L^\infty((0,T);X)}:=\operatorname*{ess\,sup}\limits_{0\leq t\leq T}\norm{v(t)}_X<\infty.$$
When there is no confusion we denote $L^p((0,T);X)$ by $L^p(X)$.
For a trilinear form $\Big(v\big(\nabla w\cdot {\bf{1}}\big),\phi\Big)$, denote by $B\big(v;w,\phi\big):=\Big(v\big(\nabla w\cdot {\bf{1}}\big),\phi\Big)$.\\
Now we present a few well known theorems and inequalities which are crucial for our analysis.
\begin{theorem}\label{fp}(\cite{kesavan})
	{\bf (Brouwer's fixed point theorem)}:  Let $H$ be a finite dimensional Hilbert space with 
	inner product $(\cdot,\cdot)$ and norm $\norm{.}.$ Let $\bF:H\rightarrow H$ be a continuous function. If there is a 
	real number  $R>0$ such that  $(\bF(z),z)>0$ $\forall~ z$ with $\norm{z}=R,$ then there exists $z^*\in H$ such that $\norm{z^*}\leq R$ and 
	$\bF(z^*)=0.$
\end{theorem}
The following trace embedding result holds for 2D.\\
{\bf{Boundary Trace Embedding Theorem} (page $164,$ \cite{adams2003}): }
There exists a bounded linear map
$$T:	H^1(\Omega)\hookrightarrow L^q(\partial \Omega) \quad \text{for}\quad 2\leq q<\infty $$
such that
\begin{equation}\label{1.30}
\norm{Ty}_{L^q(\partial \Omega)}\leq C\norm{y}_{H^1(\Omega)},
\end{equation}
for each $y\in H^1(\Omega)$.
Below, we recall  the following inequalities for our subsequent use\\
{\bf{Friedrichs's inequality}:} (See \cite{skakp2}) For $y\in H^1(\Omega),$ there holds
\begin{equation}\label{feq1}
\norm{y}^2\leq C_F\Big(\norm{\nabla y}^2+\norm{y}^2_{L^2(\partial\Omega)}\Big),
\end{equation}
where $C_F>0$ is Friedrichs's  constant. More explicitly
\begin{align*}
	\int_{\Omega} y^2dx\leq \sup_{x\in\partial\Omega}|x|^2\int_{\Omega} |\nabla y|^2 dx+\sup_{x\in\partial\Omega}|x|\int_{\partial\Omega}y^2\; d\Gamma.
\end{align*}
Hence the Friedrichs's constant can be taken as $C_F=\max\{\sup_{x\in\partial\Omega}|x|^2,\sup_{x\in\partial\Omega}|x|\}$.\\
{\bf{Gagliardo-Nireberg inequality}} (see \cite{nirenberg59}):
For $w\in H^1(\Omega)$
\begin{align*}
\norm{w}_{L^4}\leq C\Big(\norm{w}^{1/2}\norm{\nabla w}^{1/2}+\norm{w}\Big), \quad \text{and for $w\in H^2(\Omega)$}\quad
\norm{\nabla w}_{L^4}\leq C\Big(\norm{ w}^{1/4}\norm{\Delta w}^{3/4}+\norm{w}\Big).
\end{align*}
{\bf{Agmon's inequality}} (see \cite{agmons10}):
For $z\in H^2(\Omega),$ there holds
$$\norm{z}_{L^\infty}\leq C\Big(\norm{z}^\frac{1}{2}\norm{\Delta z}^\frac{1}{2}+\norm{z}\Big).$$	
\subsection{Some properties of the steady state problem}
Let $H^1/\mathbb{R}$ be the quotient space. Infact, $\norm{\nabla v}$ is a norm on this space. We now make the following assumption:\\
Assumption ${\bf (A1)}$
\begin{equation*}
\norm{f^\infty}_{-1}\leq \frac{3\nu^2}{16N},
\end{equation*} 
where $\norm{f^\infty}_{-1}=\sup_{v\in H^1/\mathbb{R}}\frac{(f^\infty,v)}{\norm{\nabla v}}$ and $N=\max(N_1,N_2),$ with
$N_1=\sup_{v,\hspace{0.1cm}z, \hspace{0.1cm}\phi \in H^1/\mathbb{R}}\frac{B\big(v;z,\phi\big)}{\norm{\nabla v}\norm{\nabla z}\norm{\nabla \phi}}$ and $N_2=\sup_{v\in H^1/\mathbb{R},\hspace{0.1cm}(z,\phi) \in (H^1(\Omega))^2}\frac{B\big(v;z,\phi\big)}{\norm{\nabla v}\tnorm{ z}\tnorm{ \phi}}$.\\
The assumption ${\bf (A1)}$ provides bound for the steady state problem \eqref{feq1.5}-\eqref{feq1.6}.
\begin{lemma}\label{key}
 Under the assumption  ${\bf (A1)}$, there exists a solution $u^\infty \in H^1/\mathbb{R}$ of \eqref{feq1.5}-\eqref{feq1.6} satisfying the following estimate:
\begin{equation}\label{fx2.1}
\norm{\nabla u^\infty}\leq \frac{\nu}{4N}.
\end{equation}
In addition, if $u^\infty\in H^2(\Omega)$ and $f^\infty\in L^2(\Omega)$, then $\norm{\Delta u^\infty}\leq C(\nu,\norm{f^\infty})$.
\end{lemma}
\begin{proof}
Multiply \eqref{feq1.5} by $u^\infty$ to obtain
\begin{align*}
\nu\norm{\nabla u^\infty}^2&=(f^\infty,u^\infty)-B\big(u^\infty; u^\infty,u^\infty\big)\\
&\leq \Big(\norm{f^\infty}_{-1}+N\norm{\nabla u^\infty}^2\Big)\norm{\nabla u^\infty}.
\end{align*}
Therefore 
\begin{align*}
\norm{\nabla u^\infty}\Big(\nu-N\norm{\nabla u^\infty}\Big)\leq \norm{f^\infty}_{-1}.
\end{align*}
Using assumption ${\bf (A1)}$, it follows that
\begin{align*}
N\norm{\nabla u^\infty}^2-\nu \norm{\nabla u^\infty}+\frac{3\nu^2}{16N}\geq 0,
\end{align*}
that is,
\begin{align*}
(\norm{\nabla u^\infty}-\frac{3\nu}{4N})(\norm{\nabla u^\infty}-\frac{\nu}{4N})\geq 0.
\end{align*}
Hence, when both the factor is negative, we obtain $\norm{\nabla u^\infty}\leq\frac{3\nu}{4N}$ and $\norm{\nabla u^\infty}\leq\frac{\nu}{4N}$.\\
Multiply \eqref{feq1.5} by $-\Delta u^\infty$ to arrive  using the bounds of $\norm{u^\infty}$ and $\norm{\nabla u^\infty}$ at
\begin{align}
\nu\norm{\Delta u^\infty}^2&=(f^\infty,-\Delta u^\infty)+B\big(u^\infty;u^\infty,\Delta u^\infty\big)\notag\\
&\leq \frac{\nu}{4}\norm{\Delta u^\infty}^2+\frac{1}{\nu}\norm{f^\infty}^2+\norm{u^\infty}_{L^4}\norm{\nabla u^\infty}_{L^4}\norm{\Delta u^\infty}\notag\\
&\leq \frac{\nu}{2}\norm{\Delta u^\infty}^2+\frac{1}{\nu}\norm{f^\infty}^2+C_{GN}\Big(\norm{u^\infty}^6\norm{\nabla u^\infty}^4+\norm{u^\infty}^3\norm{\nabla u^\infty}+\norm{u^\infty}^{10}+\norm{u^\infty}^4\Big),
\end{align}
where the constant $C_{GN}$ is appeared in the Gagliardo Nirenberg inequality.\\
Hence when $f^\infty\in L^2(\Omega)$, $\norm{\Delta u^\infty}\leq C\Big(\nu,\norm{f^\infty}\Big)$.\\
Now below, we provide the proof for the existence of a  weak solution $u^\infty$ which satisfy \eqref{fx2.1}.
Corresponding weak formulation for the steady state solution \eqref{feq1.5}-\eqref{feq1.6} is to seek $u^\infty\in H^1(\Omega)/\mathbb{R}$ such that
\begin{equation}\label{xe1}
\nu(\nabla u^\infty,\nabla v)+B\big(u^\infty; u^\infty,v\big)=(f^\infty,v) \quad \forall v\in H^1(\Omega)/\mathbb{R}.
\end{equation}
	Let $\{\phi_i\}_{i=1}^{\infty}$ be an orthogonal basis for $H^1(\Omega)/\mathbb{R}$.
Denote $V^m=span\{\phi_1,\phi_2,..., \phi_m\}$ and finite linear combination of $\phi_i$ are dense in $H^1(\Omega)/\mathbb{R}$.
We look for $u^\infty_m\in V^m$ such that
\begin{equation}\label{xe2}
\nu(\nabla u^\infty_m,\nabla v_m)+B\big(u^\infty_m;u^\infty_m,v_m\big)=(f^\infty,v_m) \quad \forall v_m\in V^m.
\end{equation}
It is sufficient to verify \eqref{xe2} for $v_m=\phi_i$, $1\leq i\leq m$.
Let $\alpha=(\alpha_i)_{i=1}^{m}\in \mathbb{R}^m$. For such $\alpha$, we relate 
a unique element $v_m\in V^m$ by $v_m=\sum_{i=1}^{m}\alpha_i\phi_i$. This establishes a linear bijection between $\mathbb{R}^m$ and $V^m$. Since 
$\{\phi_i\}$ are orthogonal in $H^1(\Omega)/\mathbb{R}$, 
$\norm{v}^2_1=\sum_{i=1}^{m}
|\alpha_i|^2=|\alpha|^2$.
Now we define $F:\mathbb{R}^m \mapsto \mathbb{R}^m$ by
 \begin{align*}
(F(\alpha))_{i}=\nu(\nabla v_m,\nabla\phi_{i})+B\big(v_m; v_m,\phi_i\big)-(f^\infty,\phi_i).
\end{align*}
	Hence \eqref{xe2} has a solution if there exists a $\alpha$ such that
$F(\alpha)=0$.
It is valid that  
\begin{align*}
(F(\alpha),\alpha)&=\sum_{i=1}^{m}(F(\alpha))_{i}\alpha_i\\& =\nu \norm{\nabla v_m}^2-B\big(v_m; v_m,v_m\big)-(f^\infty,v)\\
&\geq \norm{\nabla v_m}\Big(\nu \norm{\nabla v_m}-N\norm{\nabla v_m}^2-\norm{f^\infty}_{-1}\Big).
\end{align*}
Using the bound for the steady state solution $\norm{\nabla v_m}\leq \frac{\nu}{4N}$, we get
\begin{align*}
(F(\alpha),\alpha)\geq \norm{\nabla v_m}\Big(\frac{3\nu}{4}\norm{\nabla v_m}-\norm{f^\infty}_{-1}\Big)
\end{align*}
Hence if $|\alpha|=R$ is chosen large enough such that $\frac{3\nu}{4} R>\norm{f^\infty}_{-1}$, we have $$\Big(F(\alpha),\alpha\Big)> 0\quad \forall\quad |\alpha|=R.$$ From the condition $\norm{f^\infty}_{-1}\leq \frac{3\nu^2}{16N}$, infact we can take $R=\frac{\nu}{4N}$.
Further  by \eqref{xe2}, $\norm{u^\infty_m}_1\leq R$ and $R$ is independent of $m$.
Since $\{u^\infty_m\}$ is uniformly bounded in $ H^1(\Omega)/\mathbb{R}$, so there exists a convergent subsequence still denoted by $\{u^\infty_m\}$ such that $u^\infty_m \rightharpoonup u^\infty$ in $ H^1(\Omega)/\mathbb{R}$. Since $ H^1(\Omega)/\mathbb{R}$ is compactly embedded in $L^2$, then there exists a convergent subsequence  still denoted by $\{u^\infty_m\}$ such that $u^\infty_m \to u^\infty$ in $L^2$.\\
Hence by Brouwer's fixed point theorem \ref{fp}, there exists a $\alpha^*$ such that $\alpha^*\leq R$ and $F(\alpha^*)=0$.\\
Let $m>i$. From \eqref{xe2}, we obtain
\begin{equation*}
\nu(\nabla u^\infty_{m},\nabla \phi_{i})+(u^\infty_m(\nabla u^\infty_m\cdot {\bf{1}}),\phi_{i})=(f^\infty,\phi_{i}).
\end{equation*}
Clearly 
\begin{equation*}
\nu(\nabla u^\infty_{m},\nabla \phi_{i})\to \nu(\nabla u^\infty,\nabla \phi_{i}) \quad\text{as}\quad m\to \infty.
\end{equation*}
For the nonlinear term, we can rewrite
\begin{align*}
\Big(u^\infty_m(\nabla u^\infty_m\cdot {\bf{1}}),\phi_{i}\Big)&-\Big(u^\infty (\nabla u^\infty\cdot {\bf{1}}),\phi_{i}\Big)\\
&=\Big((u^\infty_m-u^\infty)(\nabla u^\infty_m\cdot {\bf{1}}),\phi_{i}\Big)+\Big(u^\infty\big((\nabla u^\infty_m-\nabla u^\infty)\cdot {\bf{1}}\big),\phi_{i}\Big)\\&=\Big((u^\infty_m-u^\infty)(\nabla u^\infty_m\cdot {\bf{1}}),\phi_{i}\Big)-\Big((u^\infty_m-u^\infty)(\nabla u^\infty\cdot {\bf{1}}),\phi_{i}\Big)\\
&\qquad-\Big(u^\infty(\nabla \phi_{i}\cdot {\bf{1}}),(u^\infty_m-u^\infty)\Big)-\int_{\partial\Omega}u^\infty(u^\infty_m-u^\infty)(n\cdot {\bf{1}})\phi_{i}\;d\Gamma.
\end{align*}
Boundedness of $\norm{u^\infty_m}_{1}$ and $\norm{u^\infty_m}_{L^\infty}$,  strong convergence of $u^\infty_m \to u^\infty$ in $L^2$ norm and compact embedding of $H^1(\Omega)$ onto $H^s(\partial\Omega)$ with $s<1$ implies that right hand side goes to zero in the above inequality.
Hence when $i$ goes to infinity, we obtain 
\begin{equation}\label{xe3}
\nu(\nabla u^\infty,\nabla \phi_{i})+(u^\infty(\nabla u^\infty\cdot {\bf{1}}),\phi_{i})=(f^\infty,\phi_{i}) \quad \forall i.
\end{equation}
By density, \eqref{xe3} is true for all $v\in H^1(\Omega)/\mathbb{R}$.
\end{proof}
\section{Stabilization results}
Before proceeding, let us first construct the feedback control law $v_2$.
To obtain  $v_2$, we consider Lyapunov energy functional 
$V(t)=\frac{1}{2}\int_{\Omega}w(x,t)^2\; dx$.
Then
\begin{align*}
\frac{dV}{dt}&=\int_{\Omega}w\Big(\nu \Delta w-u^\infty(\nabla w\cdot {\bf{1}})-w(\nabla u^\infty\cdot {\bf{1}})-w(\nabla w\cdot {\bf{1}})\Big)\;dx\\
&=-\nu\norm{\nabla w}^2+\nu\int_{\partial\Omega}\frac{\partial w}{\partial n}w\;d\Gamma-\int_{\Omega}\Big(u^\infty(\nabla w\cdot {\bf{1}})+w(\nabla u^\infty\cdot {\bf{1}})\Big)w\; dx-\int_{\Omega}w(\nabla w\cdot {\bf{1}})w\; dx.
\end{align*}
Using the notation $B\big(\cdot;\cdot,\cdot\big)$ for trilinear term, we can bound the right hand term as 
	\begin{align*}
&\Big(B\big(u^\infty; w,w\big)+B\big(w; u^\infty,w\big)\Big)\notag\\
&=\frac{1}{2}B\big(w;u^\infty,w\big)+\frac{1}{2}\sum_{j=1}^{2}\int_{\partial\Omega}|u^\infty|w^2\cdot\nu_j d\Gamma\notag\\
&\leq \frac{1}{2}\sum_{j=1}^{2}\int_{\partial\Omega}|u^\infty|w^2\cdot\nu_j d\Gamma+\frac{1}{2}N\norm{\nabla u^\infty}\tnorm{ w}^2\notag\\
&\leq \int_{\partial\Omega} |u^\infty| w^2 d\Gamma+\frac{1}{2}N\norm{\nabla u^\infty}\Big(\norm{\nabla w}^2+\norm{w}^2_{L^2(\partial\Omega)}\Big)\notag\\
&\leq \int_{\partial\Omega} |u^\infty| w^2 d\Gamma+\frac{\nu}{8}\Big(\norm{\nabla w}^2+\norm{w}^2_{L^2(\partial\Omega)}\Big),
\end{align*}
and
\begin{align*}
B\big(w;w,w\big)\leq \frac{1}{3}\sum_{j=1}^{2}\int_{\partial\Omega}w^3\cdot\nu_j d\Gamma&\leq \frac{1}{3}\sqrt 2\int_{\partial\Omega}|w|^3 d\Gamma\notag\\
&\leq \frac{c_0}{2}\int_{\partial\Omega}w^2 d\Gamma+\frac{1}{9c_0}\int_{\partial\Omega}w^4 d\Gamma.
\end{align*}
Hence, we get
 \begin{align*}
\frac{dV}{dt}&\leq -\frac{7\nu}{8}\norm{\nabla w}^2+\int_{\partial\Omega}\Big(\nu\frac{\partial w}{\partial n}+|u^\infty| w+(\frac{c_0}{2}+\frac{\nu}{8})w+\frac{1}{9c_0}w^3\Big)w d\Gamma.
\end{align*}
Now, choose the Neumann boundary feedback control law as
\begin{align}\label{feqx1}
\frac{\partial w}{\partial n}=v_2(x,t)=-\frac{1}{\nu}\Big((c_0+\nu+2|u^\infty|)w+\frac{2}{9c_0}w^3\Big) \quad \text{on} \quad\partial\Omega,
\end{align}
to obtain
\begin{align*}
\frac{dV}{dt}&\leq -\frac{7\nu}{8}\norm{\nabla w}^2-\int_{\partial\Omega}\Big(\frac{c_0}{2}+\frac{7\nu}{8}+|u^\infty|\Big)w^2 d\Gamma-\frac{1}{9c_0}\int_{\partial\Omega}w^4 d\Gamma\\
&\leq -\frac{1}{C_F}\min\Big\{\frac{7\nu}{8},\frac{c_0}{2}+\frac{7\nu}{8}\Big\}\norm{w}^2\leq-C_{LYP}V,
\end{align*}
$C_{LYP}=\frac{2}{C_F}\min\Big\{\frac{7\nu}{8},\frac{c_0}{2}+\frac{7\nu}{8}\Big\}$.
Now, $w$ satisfies the weak formulation of \eqref{feq1.7}-\eqref{feq1.9} as
\begin{align}
(w_t,v)+&\nu(\nabla w,\nabla v)+B\big(u^\infty;w,v\big)+B\big(w;u^\infty,v\big)+B\big(w;w,v\big)\notag\\
&+\Big<(c_0+\nu+2|u^\infty|)w+\frac{2}{9c_0}w^3,v\Big>_{\partial\Omega}\; =0\quad\forall~v\in H^1(\Omega)\label{feq1.10}
\end{align}
with $w(0)=w_0,$ where $\langle v,w\rangle_{\partial\Omega}:=\int_{\partial \Omega}vw \;d\Gamma$. Throughout the paper $C=C(\norm{w_0}_3,\nu,f^\infty)$ is a generic positive constant independent of the discretizing parameter $h$.\\
Further we make the following assumption\\
Assumption ${\bf (A2)}$
\begin{itemize}
	\item Let $u^\infty\in H^2(\Omega)$, $f^\infty\in L^2(\Omega)$.
	\item Compatibility conditions at $t=0$ $\Big(\frac{\partial w_0}{\partial n}=v_2(x,0), \frac{\partial w_t}{\partial n}(x,0)=v_{2t}(x,0)\Big)$ are satisfied.
	\item There exists a unique weak solution $w$ of \eqref{feq1.10} satisfying the following regularity result
	\begin{align*}
	\norm{w(t)}^2_2+\norm{w_t(t)}^2_1+\int_{0}^{t}\norm{ w_t(s)}^2_2 ds\leq C.
	\end{align*}
\end{itemize}
Our main objective in this section is to establish global stabilization results for the state variable $w(t)$ of the continuous problem \eqref{feq1.10}. Throughout the paper, all the results hold  under the assumptions ${\bf (A1)}$ and ${\bf (A2)}$ with the same decay rate $\alpha$
\begin{equation}\label{decay}
0\leq \alpha\leq \frac{1}{2C_F}\min\Big\{\frac{\nu}{2},c_0+\frac{7\nu}{4}\Big\}.
\end{equation}
\begin{lemma}\label{flm1}
Assume that assumption ${\bf (A1)}$	is satisfied and $u^\infty$ is a steady state solution satisfying \eqref{fx2.1} of \eqref{feq1.5}-\eqref{feq1.6}.
Let $w_0\in L^2(\Omega)$. Then, there holds
\begin{align*}
\norm{w(t)}^2&+\delta e^{-2\alpha t}\int_{0}^{t}e^{2\alpha s}\Big(\norm{\nabla w(s)}^2+\norm{w(s)}^2_{L^2(\partial\Omega)}+\frac{2}{9c_0}\norm{w(s)}^4_{L^4(\partial\Omega)} \Big)\;ds\\
&+2e^{-2\alpha t}\int_{0}^{t}e^{2\alpha s}(\int_{\partial\Omega}|u^\infty|w(s)^2\; d\Gamma) \; ds\leq e^{-2\alpha t}\norm{w_0}^2,
\end{align*}
where $\delta=\min\{\Big(\frac{7\nu}{4}-2\alpha C_F\Big), \Big(c_0+\frac{7\nu}{4}-2\alpha C_F\Big)\}>0$, and $C_F>0$ is the constant in the Friedrichs's inequality \eqref{feq1}.
\end{lemma}
\begin{proof}
	Set $v=e^{2\alpha t}w$ in \eqref{feq1.10} to obtain
	\begin{align}
	\frac{d}{dt}\norm{e^{\alpha t}w}^2-&2\alpha \norm{e^{\alpha t}w}^2+2\nu \norm{e^{\alpha t}\nabla w}^2+2e^{2\alpha t}\int_{\partial \Omega}\Big((c_0+\nu+2|u^\infty|)w^2+\frac{2}{9c_0}w^4\Big)\; d\Gamma\notag\\
	&=-2e^{2\alpha t}\Big(B\big(u^\infty; w,w\big)+B\big(w; u^\infty,w\big)\Big)-2e^{2\alpha t}B\big(w;w,w\big)
	\label{feq1.11}.
	\end{align}
	For the first term on the right hand side of \eqref{feq1.11}, we use integration by parts for the first sub-term and then bound it as follows
	\begin{align}
	2&e^{2\alpha t}\Big(B\big(u^\infty; w,w\big)+B\big(w; u^\infty,w\big)\Big)\notag\\
	&\leq 2 e^{2\alpha t}\int_{\partial\Omega} |u^\infty| w^2 d\Gamma+\frac{\nu}{4}e^{2\alpha t}\Big(\norm{\nabla w}^2+\norm{w}^2_{L^2(\partial\Omega)}\Big)\label{feq1.12}.
	\end{align}
	Similarly, using the Young's inequality, the second term on the right hand side of \eqref{feq1.11} is bounded by
	\begin{align}
	2e^{2\alpha t}B\big(w;w,w\big)\leq \frac{2}{3}e^{2\alpha t}\sum_{j=1}^{2}\int_{\partial\Omega}w^3\cdot\nu_j d\Gamma
	&\leq c_0 e^{2\alpha t}\int_{\partial\Omega}w^2 d\Gamma+\frac{2}{9c_0}e^{2\alpha t}\int_{\partial\Omega}w^4 d\Gamma\label{feq1.13}.
	\end{align}
	Now, using the Friedrichs's inequality \eqref{feq1}, it follows that
	\begin{equation}\label{feq1.14}
	-2\alpha e^{2\alpha t}\norm{w}^2 \geq -2\alpha e^{2\alpha t}C_F\Big(\norm{\nabla w}^2+\norm{w}^2_{L^2(\partial\Omega)}\Big).
	\end{equation}
		Hence, from \eqref{feq1.11}, we arrive using \eqref{feq1.12}, \eqref{feq1.13} and \eqref{feq1.14} at
	\begin{align}\label{feq1.15}
	\frac{d}{dt}\norm{e^{\alpha t}w}^2+&(\frac{7\nu}{4}-2\alpha C_F)\norm{e^{\alpha t}\nabla w}^2+e^{2\alpha t}\Big(\big(c_0+\frac{7\nu}{4}-2\alpha C_F\big)\int_{\partial\Omega}w^2 d\Gamma+2\int_{\partial\Omega}|u^\infty|w^2 d\Gamma\notag\\
	&\qquad+\frac{2}{9c_0}\int_{\partial\Omega}w^4 d\Gamma\Big)\leq 0.
	\end{align}
Choose $\alpha$ as \eqref{decay}, so that the coefficients on the left hand side of \eqref{feq1.15} are non-negative.
	Integrate \eqref{feq1.15} from $0$ to $t,$ and then, multiply the resulting inequality by $e^{-2\alpha t}$ to obtain
	\begin{align*}
	\norm{w(t)}^2+&(\frac{7\nu}{4}-2\alpha C_F)e^{-2\alpha t}\int_{0}^{t}e^{2\alpha s}\norm{\nabla w(s)}^2 ds+e^{-2\alpha t}\int_{0}^{t}e^{2\alpha s}\Big(\big(c_0+\frac{7\nu}{4}-2\alpha C_F\big)\norm{w(s)}^2_{L^2(\partial\Omega)}\\
	&\qquad+2\int_{\partial\Omega}|u^\infty|w(s)^2\; d\Gamma+\frac{2}{9c_0}\norm{w(s)}^4_{L^4(\partial\Omega)}\Big) \;ds\leq e^{-2\alpha t}\norm{w_0}^2.
	\end{align*}
	This completes the proof.
\end{proof}
\begin{remark}\label{frm2.1}
	The above Lemma also holds for $\alpha=0,$ that is,
		\begin{equation}\label{feqn1.2}
	\norm{w(t)}^2+\frac{7\nu}{4}\int_{0}^{t}\norm{\nabla w(s)}^2 ds+\int_{0}^{t}\int_{\partial\Omega}\Big(\big(c_0+\frac{7\nu}{4} +2|u^\infty|\big)w(s)^2+\frac{2}{9c_0}w(s)^4\Big) \;d\Gamma\leq \norm{w_0}^2.
	\end{equation}
	Moreover,
	using the Friedrichs's inequality, it follows that
	\begin{equation*}
	e^{-2\alpha t}\int_{0}^{t}e^{2\alpha s}\norm{w(s)}^2 ds\leq Ce^{-2\alpha t}\norm{w_0}^2.
	\end{equation*}
\end{remark}
\begin{remark}
	If we define the before mentioned Neumann control on some part of the boundary denoted by $\Gamma_N$ where measure of $\Gamma_N$ is non zero with remaining part zero Dirichlet boundary condition, still the stabilization result holds.
	For instance, consider $\partial\Omega=\Gamma_D\cup \Gamma_N$ with $\Gamma_D\cap \Gamma_N=\phi$, where $\Gamma_D$ and $\Gamma_N$ are sufficiently smooth. 
	Hence, from \eqref{feq1.11}, we arrive at
	\begin{align}
	\frac{d}{dt}\norm{e^{\alpha t}w}^2-&2\alpha \norm{e^{\alpha t}w}^2+2\nu \norm{e^{\alpha t}\nabla w}^2+2e^{2\alpha t}\int_{\Gamma_N}\Big((c_0+\nu+2|u^\infty|)w^2+\frac{2}{9c_0}w^4\Big)\; d\Gamma\notag\\
	&=-2e^{2\alpha t}\Big(B\big(u^\infty; w,w\big)+B\big(w; u^\infty,w\big)\Big)-2e^{2\alpha t}B\big(w;w,w\big)\notag\\
	&\leq 2 e^{2\alpha t}\int_{\Gamma_N}|u^\infty| w^2 d\Gamma+c_0 e^{2\alpha t}\int_{\Gamma_N}w^2 d\Gamma+\frac{2}{9c_0}e^{2\alpha t}\int_{\Gamma_N}w^4 d\Gamma\notag\\
	& \quad+\frac{\nu}{4}e^{2\alpha t}\Big(\norm{\nabla w}^2+\norm{w}^2_{L^2(\Gamma_N)}\Big)
	\label{xfeq1.11}.
	\end{align} 
	Using the Friedrichs's inequality $\norm{v}^2\leq C_F\Big(\norm{\nabla v}^2+\norm{v}^2_{L^2(\Gamma_N)}\Big)$, it follows that
	\begin{align*}
	\frac{d}{dt}\norm{e^{\alpha t}w}^2+(\frac{7\nu}{4}-2\alpha C_F)\norm{e^{\alpha t}\nabla w}^2+&e^{2\alpha t}\Big(\big(c_0+\frac{7\nu}{4}-2\alpha C_F\big)\int_{\Gamma_N}w^2 d\Gamma+2\int_{\Gamma_N}|u^\infty|w^2 d\Gamma\notag\\
	&\quad+\frac{2}{9c_0}\int_{\Gamma_N}w^4 d\Gamma\Big)
	\leq 0.
	\end{align*}
	This complete the rest of the proof for $L^2$- stabilization.
Stabilization result also holds similarly in higher order norms when control works on some part of the boundary.
\end{remark}
\begin{lemma}\label{flm2}
Assume that assumption ${\bf (A1)}$	is satisfied and $u^\infty$ is a steady state solution satisfying \eqref{fx2.1} of \eqref{feq1.5}-\eqref{feq1.6}.	Let $w_0\in H^1(\Omega).$ Then, there holds 
	\begin{align*}
	\Big(\norm{\nabla w(t)}^2&+\int_{\Omega}\frac{(c_0+\nu+2|u^\infty|)}{\nu}w(t)^2 d\Gamma+\frac{1}{9\nu c_0}\norm{w(t)}^4_{L^4(\partial\Omega)}\Big)+\nu e^{-2\alpha t}\int_{0}^{t}\norm{e^{\alpha s}\Delta w(s)}^2\; ds\\&\leq Ce^{C}e^{-2\alpha t}.
	\end{align*}	
\end{lemma}
\begin{proof}
	Form an $L^2$- inner product between \eqref{feq1.7} and $-e^{2\alpha t}\Delta w$ to obtain
	\begin{align}
	\frac{d}{dt}\norm{e^{\alpha t}\nabla w}^2-&2\alpha e^{2\alpha t}\norm{\nabla w}^2+2\nu\norm{e^{\alpha t}\Delta w}^2+\frac{2}{\nu}\int_{\partial\Omega}e^{2\alpha t}\Big((c_0+\nu+2|u^\infty|)w+\frac{2}{9c_0}w^3\Big)w_t\; d\Gamma\notag\\
	&=2e^{2\alpha t}\Big(B\big(u^\infty; w,\Delta w\big)+B\big(w; u^\infty,\Delta w\big)\Big)+2e^{2\alpha t}B(w;w,\Delta w)\label{feq1.21}.
	\end{align}
	The fourth term on the left hand side of \eqref{feq1.21} can be rewritten as
	\begin{align*}
	\frac{2}{\nu}\int_{\partial\Omega}&e^{2\alpha t}\Big((c_0+\nu+2|u^\infty|)w+\frac{2}{9c_0}w^3\Big)w_t d\Gamma\\
	&= \frac{d}{dt}\Big(\int_{\partial\Omega}\frac{(c_0+\nu+2|u^\infty|)}{\nu}w^2 d\Gamma+\frac{1}{9\nu c_0}\big(e^{2\alpha t}\norm{w}^4_{L^4(\partial\Omega)}\big)\Big)\notag\\
	&\qquad-2\alpha e^{2\alpha t}\Big(\int_{\partial\Omega}\frac{(c_0+\nu+2|u^\infty|)}{\nu}w^2 d\Gamma+\frac{1}{9\nu c_0}\norm{w}^4_{L^4(\partial\Omega)}\Big).
	\end{align*}
	The terms on the right hand side of \eqref{feq1.21} are bounded by
	\begin{align*}
	2e^{2\alpha t}&\Big(B\big(u^\infty;w,\Delta w\big)+B\big(w;u^\infty,\Delta w\big)\Big)\\
&\leq 4e^{2\alpha t}\Big(\norm{u^\infty}_{L^4}\norm{\nabla w}_{L^4}\norm{\Delta w}+\norm{w}_{L^4}\norm{\nabla u^\infty}_{L^4}\norm{\Delta w}\Big)	\\
&	\leq \frac{\nu}{2}\norm{e^{\alpha t}\Delta w}^2+Ce^{2\alpha t}\norm{ w}^2\Big(\norm{u^\infty}^8_{L^4}+\norm{u^\infty}^2_{L^4}+\norm{\nabla u^\infty}^2_{L^4}\Big)+C\norm{e^{\alpha t}\nabla w}^2\norm{\nabla u^\infty}^2_{L^4},
	\end{align*}
	and using Lemma \ref{flm1}
	\begin{align*}
	2e^{2\alpha t}B(w;w,\Delta w)&\leq Ce^{2\alpha t}\norm{w}_{L^4}\norm{\nabla w}_{L^4}\norm{\Delta w}\\
	&\leq \frac{\nu}{2}\norm{e^{\alpha t}\Delta w}^2+Ce^{2\alpha t}\norm{w}^2\norm{\nabla w}^4+Ce^{2\alpha t}\norm{w}^2+Ce^{2\alpha t}\norm{w}^2\norm{\nabla w}^2.
	\end{align*}
	Finally, using the bounds of $\norm{u^\infty}_2$, we arrive from \eqref{feq1.21} at
	\begin{align}
	\frac{d}{dt}\Big(e^{2\alpha t}\big(\norm{\nabla w}^2&+\int_{\Omega}\frac{(c_0+\nu+2|u^\infty|)}{\nu}w^2 d\Gamma+\frac{1}{9\nu c_0}\norm{w}^4_{L^4(\partial\Omega)}\big)\Big)+\nu\norm{e^{\alpha t}\Delta w}^2\notag\\
	&\leq 2\alpha e^{2\alpha t}\Big(\int_{\partial\Omega}\frac{(c_0+\nu+2|u^\infty|)}{\nu}w^2\; d\Gamma+
	\frac{1}{9\nu c_0}\norm{w}^4_{L^4(\partial\Omega)}\Big)+e^{2\alpha t}\norm{\nabla w}^2\notag\\ 
	&\qquad +Ce^{2\alpha t}\norm{w}^2+Ce^{2\alpha t}\norm{w}^2\norm{\nabla w}^2
	+Ce^{2\alpha t}\norm{w}^2\norm{\nabla w}^4\label{feq1.22}.
	\end{align}
	Integrate the above inequality from $0$ to $t,$ and then use the Gr\"onwall's inequality with Lemma \ref{flm1} to obtain
	\begin{align*}
	e^{2\alpha t}\Big(\norm{\nabla w(t)}^2&+\int_{\Omega}\frac{(c_0+\nu+2|u^\infty|)}{\nu}w(t)^2 d\Gamma+\frac{1}{9\nu c_0}\norm{w(t)}^4_{L^4(\partial\Omega)}\Big)+\nu\int_{0}^{t}\norm{e^{\alpha s}\Delta w(s)}^2\; ds\\&\leq C\Big(\norm{w_0}^2_1+\norm{w_0}^2_{L^2(\partial\Omega)}+\norm{w_0}^4_{L^4(\partial\Omega)}\Big)\exp\Big(C\int_{0}^{t}\norm{w}^2\norm{\nabla w}^2 ds\Big).
	\end{align*}
	Use Remark \ref{frm2.1} for the integral term under the exponential sign, and then multiply the resulting inequality by $e^{-2\alpha t}$ to complete the rest of the proof.
\end{proof}
\begin{lemma}\label{flm3}
	Assume that assumption ${\bf (A1)}$	is satisfied and $u^\infty$ is a steady state solution satisfying \eqref{fx2.1} of \eqref{feq1.5}-\eqref{feq1.6}. Let $w_0\in H^1(\Omega)$. Then,  the following estimate holds
	\begin{align*}
	\Big(\nu \norm{\nabla w(t)}^2&+\int_{\Omega}(c_0+\nu+2|u^\infty|)w(t)^2\; d\Gamma+\frac{1}{9c_0}\norm{w(t)}^4_{L^4(\partial\Omega)}\Big)+e^{-2\alpha t}\int_{0}^{t}e^{2\alpha s}\norm{w_t(s)}^2 ds\leq Ce^{C}e^{-2\alpha t}.
	\end{align*}
\end{lemma}
\begin{proof}
	Choose $v=e^{2\alpha t}w_t$ in \eqref{feq1.10} to obtain
	\begin{align}
	2\norm{e^{\alpha t}w_t}^2+\nu \frac{d}{dt}\norm{e^{\alpha t}\nabla w}^2&-2\nu\alpha\norm{e^{\alpha t}\nabla w}^2+2\int_{\partial\Omega}\Big((c_0+\nu+2|u^\infty|)w+\frac{2}{9c_0}w^3\Big)e^{2\alpha t}w_t\; d\Gamma\notag\\
	&=-2e^{2\alpha t}\Big(B\big(u^\infty;w, w_t\big)+B\big(w;u^\infty, w_t\big)\Big)-2e^{2\alpha t}B\big(w;w,w_t\big)\label{feq1.31}.
	\end{align}
	The terms on the right hand side of \eqref{feq1.31} are bounded by
	\begin{align*}
	2e^{2\alpha t}\Big(B\big(u^\infty;w, w_t\big)+B\big(w;u^\infty, w_t\big)\Big)&\leq \frac{1}{2}e^{2\alpha t}\norm{w_t}^2+Ce^{2\alpha t}\Big(\norm{ w}^2+\norm{\nabla w}^2\Big)\norm{\nabla u^\infty}^2_{L^4}\\
	&\qquad+Ce^{2\alpha t}\Big(\norm{ w}^2+\norm{\Delta w}^2\Big)\norm{ u^\infty}^2_{L^4},
	\end{align*}
	and using Lemma \ref{flm1}
	\begin{align*}
	2e^{2\alpha t}B\big(w;w,w_t\big)&\leq Ce^{2\alpha t}\norm{w}_{L^4}\norm{\nabla w}_{L^4}\norm{w_t}\\
	&\leq \frac{1}{2}e^{2\alpha t}\norm{w_t}^2+Ce^{2\alpha t}\norm{w}^2\norm{\nabla w}^4+Ce^{2\alpha t}\norm{\Delta w}^2+Ce^{2\alpha t}\norm{w}^2\\
	&\qquad+Ce^{2\alpha t}\norm{w}^2\norm{\nabla w}^2.
	\end{align*}
	Hence, rewriting the boundary integral term in \eqref{feq1.31} as in previous Lemma \ref{flm2}, we arrive from \eqref{feq1.31} at
	\begin{align*}
	\frac{d}{dt}\Big(e^{2\alpha t}\big(\nu \norm{\nabla w}^2&+\int_{\Omega}(c_0+\nu+2|u^\infty|)w^2\; d\Gamma+\frac{1}{9c_0}\norm{w}^4_{L^4(\partial\Omega)}\big)\Big)+\norm{e^{\alpha t}w_t}^2\\
	&\leq Ce^{2\alpha t}\Big(\int_{\Omega}(c_0+\nu+2|u^\infty|)w^2\; d\Gamma+\frac{1}{9c_0}\norm{w}^4_{L^4(\partial\Omega)}+\norm{ w}^2_2+\norm{w}^2\norm{\nabla w}^2\\
	&\hspace{2cm}+\norm{w}^2\norm{\nabla w}^4+\norm{w}^2_2\big(\norm{u^\infty}^2_{L^4}+\norm{\nabla u^\infty}^2_{L^4}\big)\Big).
	\end{align*}
	Apply Lemmas \ref{flm1} and \ref{flm2}, and the Gr\"onwall's inequality to the above inequality to complete the rest of the proof.
\end{proof}
\begin{lemma}\label{flm4}
Assume that assumptions ${\bf (A1)}$ and ${\bf (A2)}$	are satisfied and $u^\infty$ is a steady state solution satisfying \eqref{fx2.1} of \eqref{feq1.5}-\eqref{feq1.6}. Then, there holds
	\begin{align*}
	\norm{w_t(t)}^2+\nu e^{-2\alpha t}\int_{0}^{t}e^{2\alpha s}\norm{\nabla w_t(s)}^2 ds&+2e^{-2\alpha t}\int_{0}^{t}e^{2\alpha s}\Big(\int_{\partial\Omega}(c_0+\nu+2|u^\infty|)w_t(s)^2\; d\Gamma\\
	&\qquad+\frac{2}{3c_0}\norm{w(s)w_t(s)}^2_{L^2(\partial\Omega)}\Big)\; ds\leq Ce^{C}e^{-2\alpha t}.
	\end{align*}
\end{lemma}
\begin{proof}
	Differentiate \eqref{feq1.7} with respect to $t$ and then take the inner product with $e^{2\alpha t}w_t$ to obtain 
	\begin{align}
	\frac{d}{dt}\big(\norm{e^{\alpha t}w_t}^2\big)&-2\alpha \norm{e^{\alpha t}w_t}^2+2\nu\norm{e^{\alpha t}\nabla w_t}^2+2\int_{\partial\Omega}\big((c_0+\nu+2|u^\infty|)w_t^2+\frac{2}{3c_0}w^2w_t^2\big)e^{2\alpha t}\;d\Gamma\notag\\
	&=-2e^{2\alpha t}\Big(B\big(w_t;w,w_t\big)+B\big(w;w_t,w_t\big)\Big)-2e^{2\alpha t}\Big(u^\infty(\nabla w_t\cdot{\bf{1}})+w_t(\nabla u^\infty\cdot{\bf{1}}),w_t\Big)\label{feq2.1}.
	\end{align}
	The first right hand side term in \eqref{feq2.1} is bounded by
	\begin{align*}
	-2e^{2\alpha t}&\Big(B\big(w_t;w,w_t\big)+B\big(w;w_t,w_t\big)\Big)\\
	&\leq \frac{\nu}{2}\norm{e^{\alpha t}\nabla w_t}^2+Ce^{2\alpha t}\Big(\norm{w_t}^2+\norm{w_t}^2\norm{\nabla w}^2+\norm{w}^2\norm{\nabla w}^2\norm{w_t}^2+\norm{w}^2\norm{\nabla w}^2\\
	&\hspace{1cm}+\norm{w_t}^2\norm{w}^2+\norm{w_t}^2\norm{w}^4+\norm{w}^2\Big).
	\end{align*}
	The other right hand term in \eqref{feq2.1} is bounded by
	\begin{align*}
	-2e^{2\alpha t}&\Big(u^\infty\big(\nabla w_t\cdot{\bf{1}}\big)+w_t\big(\nabla u^\infty\cdot{\bf{1}}\big),w_t\Big)\\
	&\leq \frac{\nu}{2}\norm{e^{\alpha t}\nabla w_t}^2+Ce^{2\alpha t}\norm{w_t}^2\Big(\norm{\nabla u^\infty}^2+\norm{u^\infty}^2_{L^4}+\norm{u^\infty}^4_{L^4}\Big).
	\end{align*}
Therefore, from \eqref{feq2.1}, we obtain
	\begin{align}
	\frac{d}{dt}(\norm{e^{\alpha t}w_t}^2)+&\nu \norm{e^{\alpha t}\nabla w_t}^2+2e^{2\alpha t}
	\Big(\int_{\partial\Omega}(c_0+\nu+2|u^\infty|)w_t^2\; d\Gamma+\frac{2}{3c_0}\norm{ww_t}^2_{L^2(\partial\Omega)}\Big)\notag\\
	&\leq Ce^{2\alpha t}\Big(\norm{w_t}^2+\norm{w_t}^2\norm{\nabla w}^2+\norm{w}^2\norm{\nabla w}^2\norm{w_t}^2+\norm{w}^2\norm{\nabla w}^2\notag\\
	&\hspace{4cm}+\norm{w_t}^2\norm{w}^2+\norm{w_t}^2\norm{w}^4+\norm{w}^2\Big)\label{feq2.2}.
	\end{align}
	To calculate $\norm{w_t(0)},$
	take the inner product between \eqref{feq1.7} and $w_t$ to obtain at $t=0$
	\begin{align*}
	\norm{w_t(0)}^2\leq C\Big(\norm{\nabla w_0}^2+\norm{\Delta w_0}^2+\norm{w_0}^2\norm{\nabla w_0}^4\Big).
	\end{align*}
	Integrate the inequality \eqref{feq2.2} from $0$ to $t$ and then use Lemmas \ref{flm1}-\ref{flm3} to complete the rest of the proof.
\end{proof}
\begin{lemma}\label{flm5}
Assume that assumptions ${\bf (A1)}$ and ${\bf (A2)}$	are satisfied and $u^\infty$ is a steady state solution satisfying \eqref{fx2.1} of \eqref{feq1.5}-\eqref{feq1.6}. Then the following estimate holds
	\begin{align*}
	\Big(\norm{\nabla w_t(t)}^2+\int_{\partial\Omega}(c_0+\nu+2|u^\infty|)w_t(t)^2 \; d\Gamma&+\frac{2}{3c_0}\norm{w(t)w_t(t)}^2_{L^2(\partial\Omega)}\Big)+\nu e^{-2\alpha t}\int_{0}^{t}e^{2\alpha s}\norm{\Delta w_t(s)}^2 ds\\
	&\leq Ce^{C}e^{-2\alpha t}.
	\end{align*}
\end{lemma}
\begin{proof}
	Differentiate \eqref{feq1.7} with respect to $t$ and then take inner product with $-e^{2\alpha t}\Delta w_t$ to obtain
	\begin{align}
	\frac{d}{dt}\norm{e^{\alpha t}\nabla w_t}^2&-2\alpha\norm{e^{\alpha t}\nabla w_t}^2+2\nu\norm{e^{\alpha t}\Delta w_t}^2+\frac{d}{dt}\int_{\partial\Omega}e^{2\alpha t}\Big((c_0+\nu+2|u^\infty|)w_t^2+\frac{2}{3c_0}w^2w_t^2\Big) d\Gamma\notag\\
	&\leq 2e^{2\alpha t}\Big(B\big(u^\infty;w_t,\Delta w_t\big)+B\big(w_t;u^\infty,\Delta w_t\big)\Big)+2e^{2\alpha t}B\big(w_t,w,\Delta w_t\big)\notag\\
	&\qquad+2e^{2\alpha t}B\big(w;w_t,\Delta w_t\big)+C\int_{\partial\Omega}e^{2\alpha t}\Big(w_t^2+ww_t^3+w^2w_t^2\Big)\;d\Gamma\label{feq2.5}.
	\end{align}
	The first three terms on the right hand side of \eqref{feq2.5} are bounded by
	\begin{align*}
	2e^{2\alpha t}\Big(B\big(u^\infty;w_t,\Delta w_t\big)+B\big(w_t;u^\infty,\Delta w_t\big)\Big)&\leq \frac{\nu}{3}\norm{e^{\alpha t}\Delta w_t}^2+Ce^{2\alpha t}\norm{ w_t}^2\Big(\norm{\nabla u^\infty}^2_{L^4}+\norm{u^\infty}^2_{L^4}\Big)\\
	&\quad+Ce^{2\alpha t}\norm{\nabla w_t}^2\Big(\norm{\nabla u^\infty}^2_{L^4}+\norm{ u^\infty}^2_{L^4}\Big),
	\end{align*}
	and using Lemma \ref{flm1}
	\begin{align*}
	2e^{2\alpha t}&\Big(B\big(w_t;w,\Delta w_t\big)+
	B\big(w;w_t,\Delta w_t\big)\Big)\\
	&\leq \frac{2\nu}{3}\norm{e^{\alpha t}\Delta w_t}^2+Ce^{2\alpha t}\norm{w_t}^2\Big(\norm{w}^2_2+\norm{w}^2\norm{\nabla w}^4\Big)+Ce^{2\alpha t}\norm{\nabla w_t}^2\Big(\norm{w}^2+\norm{\Delta w}^2\Big).
	\end{align*}
	The boundary term on the right hand side of \eqref{feq2.5} is bounded by
	\begin{align*}
	C\int_{\partial\Omega}e^{2\alpha t}\Big(w_t^2+ww_t^3+w^2w_t^2\Big)\;d\Gamma&\leq C\int_{\partial\Omega}e^{2\alpha t}\Big(w_t^2+w^2w_t^2\Big)\;d\Gamma+Ce^{2\alpha t}\norm{w_t}^4_{L^4(\partial\Omega)}
	.
	\end{align*}
	Hence, from \eqref{feq2.5}, we arrive at
	\begin{align*}
	\frac{d}{dt}\Big(\norm{e^{\alpha t}\nabla w_t}^2&+e^{2\alpha t}\int_{\partial\Omega}(c_0+\nu+2|u^\infty|)w_t^2\;d\Gamma+\frac{2}{3c_0}\int_{\partial\Omega}e^{2\alpha t}w^2w_t^2\;d\Gamma\Big)+\nu\norm{e^{\alpha t}\Delta w_t}^2\\
	&\leq Ce^{2\alpha t}\norm{w_t}^2\Big(\norm{w}^2_2+\norm{w}^2\norm{\nabla w}^4\Big)+Ce^{2\alpha t}\norm{\nabla w_t}^2\Big(1+\norm{w}^2+\norm{\Delta w}^2\Big)\\
	&\qquad+\int_{\partial\Omega}e^{2\alpha t}\Big(w_t^2+w^2w_t^2\Big)\;d\Gamma+Ce^{2\alpha t}\Big(\norm{w_t}^4+\norm{\nabla w_t}^4\Big).
	\end{align*}
	Integrate the above inequality from $0$ to $t$ and then apply the Gr\"onwall's inequality along with Lemmas \ref{flm1}-\ref{flm4} to obtain
	\begin{align}
	\Big(\norm{e^{\alpha t}\nabla w_t(t)}^2&+e^{2\alpha t}\int_{\partial\Omega}(c_0+\nu+2|u^\infty|)w_t(t)^2\;d\Gamma+\frac{2}{3c_0}\int_{\partial\Omega}e^{2\alpha t}w(t)^2w_t(t)^2\;d\Gamma\Big)\notag\\
	&+\nu \int_{0}^{t}\norm{e^{\alpha s}\Delta w_t(s)}^2 ds\leq C\Big(\norm{\nabla w_t(0)}^2+\norm{w_t(0)}_{L^2(\partial\Omega)}+\norm{w(0)w_t(0)}^2_{L^2(\partial\Omega)}\Big)\notag\\
	&\hspace{4.5cm}\exp\Big(C\int_{0}^{t}\big(\norm{w(t)}^2+\norm{\Delta w(t)}^2+\norm{\nabla w_t(t)}^2\big)\;ds\Big)\label{feqx2.5}.
	\end{align}
	Differentiate \eqref{feq1.7} with respect to $x_1$ and $x_2$ to get
	$\norm{\nabla w_t(0)}\leq C\norm{w_0}_3$. Further, by boundary trace embedding theorem, $\norm{w_t(0)}_{L^2(\partial\Omega)}\leq C\norm{w_t(0)}^2_1$ and 
	$$\norm{w(0)w_t(0)}^2_{L^2(\partial\Omega)}\leq C\norm{w(0)}_{L^4(\partial\Omega)}\norm{w_t(0)}^2_1.$$
	Again, using of Lemmas \ref{flm1}, \ref{flm2} and \ref{flm4} to \eqref{feqx2.5} completes the proof.
\end{proof}
\section{Finite element method}
In this section, we discuss semidiscrete Galerkin approximation keeping time variable continuous and prove
optimal error estimates for both state variable and feedback controller.\\
Given a regular triangulation  $\mathcal{T}_h$ of $\overline{\Omega},$
let $h_K=\text{diam}(K)$ for all $K\in \mathcal{T}_h$ and $h=\displaystyle\max_{ K\in \mathcal{T}_h} h_K$.\\
Set
$$V_h=\left\{v_h\in C^0(\overline{\Omega} ):\hspace{0.1cm} v_h\Big|_K \in \mathcal{P}_{1}(K) \quad \forall~\quad K\in \mathcal{T}_{h}\right\}.$$
The discrete weak formulation of the corresponding steady state solution is to seek some approximation of $u^\infty$ as $u^\infty_h\in V_h$ such that 
\begin{align}
\nu(\nabla u^\infty_h,\nabla \chi)+B\big(u^\infty_h;u^\infty_h,\chi\big)=(f^\infty,\chi) \quad \forall \quad \chi\in V_h.
\end{align}
Corresponding steady state solution for the discrete problem satisfies
\begin{align}
-\nu\Delta_hu^\infty_h&+u^\infty_h(\nabla u^\infty_h\cdot {\bf {1}})=f^\infty\quad \text{in} \quad \Omega\label{sd1}\\
\frac{\partial  u^\infty_h}{\partial n}&=0 \quad \text{on} \quad \partial\Omega\label{sd2},
\end{align}
where {\it{discrete Laplacian}} $-\Delta_h u^\infty_h:V_h\longrightarrow V_h$ is defined by
\begin{align}\label{sd3}
(-\Delta_h u^\infty_h,w_h)=(\nabla u^\infty_h,\nabla w_h)+\langle\frac{\partial u^\infty_h}{\partial n},w_h\rangle\qquad \forall~ u^\infty_h,~ w_h\in V_h.
\end{align}
Note that second term in \eqref{sd3} is zero, but we keep it to attach a meaning of {\it{discrete Laplacian}} in \eqref{feq4.1} for general nonhomogeneous boundary condition.
For $u^\infty_h=\sum_{j=1}^{N}\alpha_j\phi_j$ where $\{\phi_j\}_{j=1}^{N_h}$ are the basis functions for finite element space $V_h$, the discrete normal derivative $\frac{\partial u^\infty_h}{\partial n}$ is defined as 
$\frac{\partial u^\infty_h}{\partial n}= \sum_{j=1}^{N}  \alpha_j(\nabla \phi_j\cdot n)$.
In terms of a basis function $\{\phi_j\}_{j=1}^{N_h}$, the above analogue \eqref{sd3} of Green's formula defines $-\Delta_h u^\infty_h=\sum_{j=1}^{N_h}d_j\phi_j$ by
$$\sum_{j=1}^{N_h}d_j(\phi_j,\phi_k)=(\nabla u^\infty_h,\nabla \phi_k)+\langle\frac{\partial u^\infty_h}{\partial n},\phi_k\rangle \quad \text{for} \quad k=1,\hdots,N_h.$$
See \cite{thomee} for more details.
As in continuous case similarly the following bound holds for $u^\infty_h$ under the assumption ${\bf (A1)}$.
\begin{lemma}
Under the assumption ${\bf (A1)}$, there exists a solution $u^\infty_h\in V_h/\mathbb{R}$
to the problem \eqref{sd1}-\eqref{sd2} satisfying
\begin{equation}\label{sd4}
\norm{\nabla u^\infty_h}\leq \frac{\nu}{4N}.
\end{equation}
Moreover, when $\Delta_h u^\infty_h\in V_h$ and $f^\infty\in L^2$, then $\norm{\Delta_h u^\infty_h}\leq C(\nu,\norm{f^\infty})$.
\end{lemma}
The semidiscrete approximation corresponding to the problem  \eqref{feq1.10} is to seek $w_h(t)=w_h(\cdot,t)\in V_h$ such that
\begin{align}
(w_{ht},\chi)+&\nu(\nabla w_h,\nabla \chi)+B\big(u^\infty_h;w_h,\chi\big)+B\big(w_h;u^\infty_h,\chi\big)+B\big(w_h;w_h,\chi\big)\notag\\
&\qquad+\int_{\partial \Omega}\Big((c_0+\nu+2|u^\infty_h|)w_h+\frac{2}{9c_0}w_h^3\Big)\chi\; d\Gamma
=0,\quad \forall~ \chi\in V_h\label{feqn3.1}
\end{align}
with $w_h(0)=P_h u_0-u^\infty_h=w_{0h}$ (say), an approximation of $w_0,$
where, $P_hu_0$ is the $H^1$ projection of $u_0$ onto $V_h$ such that
\begin{equation}
\norm{u_0-u_{0h}}_j\leq Ch^{2-j}\norm{u_0}_2\quad j=0,1.
\end{equation}
Since $V_h$ is finite dimensional, \eqref{feqn3.1} leads to a system of nonlinear
ODEs. Therefore due to Picard's theorem there exists a unique solution $w_h$ locally i.e for $t\in (0,t_n)$. Moreover applying Lemmas \ref{flm3.1}-\ref{flm3.2} from below, the continuation arguments
yields existence of a unique solution for all $t>0$.
Following stabilization results hold for the semidiscrete solution in a similar approach as in continuous case.
\begin{lemma}\label{flm3.1}
Assume that  assumption	$\bf(A1)$ is satisfied and discrete steady state solution $u^\infty_h$ of \eqref{sd1}-\eqref{sd2} satisfy \eqref{sd4}. Let $w_0\in L^2(\Omega)$. Then,  there holds
	\begin{align*}
\norm{w_h(t)}^2&+\delta e^{-2\alpha t}\int_{0}^{t}e^{2\alpha s}\Big(\norm{\nabla w_h(s)}^2+\norm{w_h(s)}^2_{L^2(\partial\Omega)}+\frac{2}{9c_0}\norm{w_h(s)}^4_{L^4(\partial\Omega)} \Big)\;ds\\
&+2e^{-2\alpha t}\int_{0}^{t}e^{2\alpha s}(\int_{\partial\Omega}|u^\infty_h|w_h(s)^2\; d\Gamma) \; ds\leq e^{-2\alpha t}\norm{w_0}^2,
	\end{align*}
	where  $\delta$ and $\beta$ are same as in the continuous case.
\end{lemma}
Now the corresponding semidiscrete problem satisfies
\begin{align}
& w_{ht}-\nu \Delta_h w_h+u^\infty_h(\nabla w_h\cdot {\bf {1}})+w_h(\nabla u^\infty_h\cdot {\bf {1}})+w_h(\nabla w_h\cdot {\bf {1}})=0\quad \text{in} \quad\Omega\times(0,\infty)\label{feq4.1}\\
&\frac{\partial  w_h}{\partial n}(x,t)=-\frac{1}{\nu}\Big((c_0+\nu+2|u^\infty_h|)w_h+\frac{2}{9c_0}w^3_h\Big)=:v_{0h}(t) \quad \text{on} \quad\partial\Omega\times(0,\infty)\label{feq4.2},\\
&w_h(0)=w_{0h} \quad \text{in} \quad\Omega\label{feq4.3}.
\end{align}
\begin{lemma}\label{flm3.2}
Assume that  assumption 	$\bf(A1)$ is satisfied and discrete steady state solution $u^\infty_h$ of \eqref{sd1}-\eqref{sd2} satisfy \eqref{sd4}. Let $w_0\in H^1(\Omega).$ Then, there holds 
	\begin{align*}
	\big(\norm{\nabla w_h(t)}^2&+\int_{\partial\Omega}\frac{(c_0+\nu+2|u^\infty_h|)}{\nu}w_h(t)^2\;d\Gamma+\frac{1}{9\nu c_0}\norm{w_h(t)}^4_{L^4(\partial\Omega)}\big)+\nu e^{-2\alpha t}\int_{0}^{t}\norm{e^{\alpha s}\Delta_h w_h(s)}^2\; ds\\&\leq Ce^{C}e^{-2\alpha t}.
	\end{align*}	
\end{lemma}
\begin{proof}
	Forming an $L^2$- inner product between \eqref{feq4.1} and $-e^{2\alpha t}\Delta_h w_h,$ we obtain
	\begin{align}
	\frac{d}{dt}\norm{e^{\alpha t}\nabla w_h}^2-&2\alpha e^{2\alpha t}\norm{\nabla w_h}^2+2\nu\norm{e^{\alpha t}\Delta_h w_h}^2+\frac{2}{\nu}\int_{\partial\Omega}e^{2\alpha t}\Big((c_0+\nu+2|u^\infty_h|)w_h+\frac{2}{9c_0}w_h^3\Big)w_{ht} d\Gamma\notag\\
	&=2e^{2\alpha t}(u^\infty_h(\nabla w_h\cdot {\bf{1}})+w_h(\nabla u^\infty_h\cdot {\bf{1}}),\Delta_h w_h)+2e^{2\alpha t}B(w_h;w_h,\Delta_h w_h)\label{feq4.21}.
	\end{align}
	Finally bounding the right hand side term, from \eqref{feq4.21}, it follows that
	\begin{align*}
	\frac{d}{dt}\Big(e^{2\alpha t}\big(\norm{\nabla w_h}^2&+\int_{\partial\Omega}\frac{(c_0+\nu+2|u^\infty_h|)}{\nu}\norm{w_h}^2 \; d\Gamma+\frac{1}{9\nu c_0}\norm{w_h}^4_{L^4(\partial\Omega)}\big)\Big)+\nu\norm{e^{\alpha t}\Delta_h w_h}^2\\
	&\leq 2\alpha e^{2\alpha t}\Big(\int_{\partial\Omega}\frac{(c_0+\nu+2|u^\infty_h|)}{\nu}w_h^2 \; d\Gamma+
	\frac{1}{9\nu c_0}\norm{w_h}^4_{L^4(\partial\Omega)}\Big)+Ce^{2\alpha t}\norm{\nabla w_h}^2\\
	&\qquad +Ce^{2\alpha t}\norm{w_h}^{2}+Ce^{2\alpha t}\norm{w_h}^2\norm{\nabla w_h}^2
	+Ce^{2\alpha t}\norm{w_h}^2\norm{\nabla w_h}^4.
	\end{align*}
	Integrate from $0$ to $t,$ and then use the Gr\"onwall's inequality with Lemma \ref{flm3.1} to obtain
	\begin{align*}
	e^{2\alpha t}\big(\norm{\nabla w_h(t)}^2&+\int_{\partial\Omega}\frac{(c_0+\nu+2|u^\infty_h|)}{\nu}w_h(t)^2\; d\Gamma+\frac{1}{9\nu c_0}\norm{w_h(t)}^4_{L^4(\partial\Omega)}\big)+\nu\int_{0}^{t}\norm{e^{\alpha s}\Delta_h w_h(s)}^2\; ds\\&\leq C(\norm{w_0}^2_1+\norm{w_0}^2_{L^2(\partial\Omega)}+\norm{w_0}^4_{L^4(\partial\Omega)})\exp\Big(C\int_{0}^{t}\norm{w_h(s)}^2\norm{\nabla w_h(s)}^2\big) ds\Big).
	\end{align*}
	Apply Lemma \ref{flm3.1} for $\alpha=0$ to the integral term under the exponential sign, and then multiply the resulting inequality by $e^{-2\alpha t}$ to complete the rest of the proof.
\end{proof}
\begin{lemma}\label{flm3.3}
Assume that assumption 	$\bf(A1)$ is satisfied and discrete steady state solution $u^\infty_h$ of \eqref{sd1}-\eqref{sd2} satisfy \eqref{sd4}. Let $w_0\in H^1(\Omega)$. Then,  there holds
	\begin{align*}
	\big(\nu \norm{\nabla w_h(t)}^2&+\int_{\partial\Omega}(c_0+\nu+2|u^\infty_h|)w_h(t)^2\;d\Gamma+\frac{1}{9c_0}\norm{w_h(t)}^4_{L^4(\partial\Omega)}\big)\\
	&\quad+e^{-2\alpha t}\int_{0}^{t}e^{2\alpha s}\norm{w_{ht}(s)}^2 ds\leq Ce^{C}e^{-2\alpha t}.
	\end{align*}
\end{lemma}
\begin{proof}
	Proof follows as in continuous case, namely; Lemma \ref{flm3}.
\end{proof}
\begin{lemma}\label{flm3.4}
	Assume that assumption 	$\bf(A1)$ is satisfied and discrete steady state solution $u^\infty_h$ of \eqref{sd1}-\eqref{sd2} satisfy \eqref{sd4}. Let $w_0\in H^2(\Omega)$. Then, we get
	\begin{align*}
	\norm{w_{ht}(t)}^2+&\nu e^{-2\alpha t}\int_{0}^{t}e^{2\alpha s}\norm{\nabla w_{ht}(s)}^2 ds+2e^{-2\alpha t}\int_{0}^{t}e^{2\alpha s}\Big(\int_{\partial\Omega}(c_0+\nu+2|u^\infty_h|)w_{ht}(s)^2\;d\Gamma\\
	&\hspace{3cm}+\frac{2}{3c_0}\norm{w_h(s)w_{ht}(s)}^2_{L^2(\partial\Omega)}\Big)\; ds\leq Ce^{C}e^{-2\alpha t}.
	\end{align*}
\end{lemma}
\begin{proof}
	Differentiating \eqref{feq4.1} with respect to t and then taking inner product with $e^{2\alpha t}w_{ht},$ we obtain 
	\begin{align}
	\frac{d}{dt}\big(\norm{e^{\alpha t}w_{ht}}^2\big)-&2\alpha \norm{e^{\alpha t}w_{ht}}^2+2\nu\norm{e^{\alpha t}\nabla w_{ht}}^2+2\int_{\partial\Omega}\big((c_0+\nu+2|u^\infty_h|)w_{ht}^2+\frac{2}{3c_0}w_h^2w_{ht}^2\big)e^{2\alpha t}\;d\Gamma\notag\\
	&= -2e^{2\alpha t}\Big(B\big(w_{ht};w_h,w_{ht}\big)+B\big(w_h;w_{ht},w_{ht}\big)\Big)\notag\\
	&\qquad-2e^{2\alpha t}\Big(B\big(u^\infty_h;w_{ht},w_{ht}\big)+B\big(w_{ht};u^\infty_h,w_{ht}\big)\Big)\label{feq4.23}.
	\end{align}
Bounding the right hand side term of \eqref{feq4.23}, it follows that
	\begin{align}
	\frac{d}{dt}(\norm{e^{\alpha t}w_{ht}}^2)+&\nu \norm{e^{\alpha t}\nabla w_{ht}}^2+2e^{2\alpha t}
	\Big(\int_{\partial\Omega}(c_0+\nu+2|u^\infty_h|)w_{ht}^2\;d\Gamma+\frac{2}{3c_0}\norm{w_hw_{ht}}^2_{L^2(\partial\Omega)}\Big)\notag\\
	&\leq Ce^{2\alpha t}\Big(\norm{w_{ht}}^2+\norm{w_{ht}}^2\norm{ w_h}^2_1+\norm{w_h}^2\norm{\nabla w_h}^2\norm{w_{ht}}^2\notag\\
	&\hspace{2cm}+\norm{w_h}^2\norm{\nabla w_h}^2+\norm{w_{ht}}^2\norm{w_h}^4+\norm{w_h}^2\Big)\label{feq4.24}.
	\end{align}
	To obtain $\norm{w_{ht}(0)},$
	take the inner product between \eqref{feq4.1} and $w_{ht}$ to arrive at
	\begin{align*}
	\norm{w_{ht}(0)}^2\leq C\Big(\norm{\nabla w_{0h}}^2+\norm{\Delta_h w_{0h}}^2+\norm{w_{0h}}^2\norm{\nabla w_{0h}}^4\Big).
	\end{align*}
	Integrate the inequality \eqref{feq4.24} from $0$ to $t$ and then use Lemmas \ref{flm3.1}-\ref{flm3.3} to complete the proof.
\end{proof}
\begin{lemma}\label{flm3.5}
Assume that assumption 	$\bf(A1)$ is satisfied and discrete steady state solution $u^\infty_h$ of \eqref{sd1}-\eqref{sd2} satisfy \eqref{sd4}. Let $w_0\in H^3(\Omega).$ Then, we obtain
	\begin{align*}
	\norm{\nabla w_{ht}(t)}^2+&\big(\int_{\partial\Omega}(c_0+\nu+2|u^\infty_h|)w_{ht}(t)^2\;d\Gamma+\frac{2}{3c_0}\norm{w_h(t)w_{ht}(t)}^2_{L^2(\partial\Omega)}\big)\\
	&\qquad+\nu e^{-2\alpha t}\int_{0}^{t}e^{2\alpha s}\norm{\Delta_h w_{ht}(s)}^2 ds\leq Ce^{C}e^{-2\alpha t}.
	\end{align*}
\end{lemma}
\begin{proof}
	Differentiate \eqref{feq4.1} with respect to $t$ and then take inner product with $-e^{2\alpha t}\Delta_h w_{ht}$ to obtain
	\begin{align}
	\frac{d}{dt}\norm{e^{\alpha t}\nabla w_{ht}}^2&-2\alpha\norm{e^{\alpha t}\nabla w_{ht}}^2+2\nu\norm{e^{\alpha t}\Delta_h w_{ht}}^2+\frac{d}{dt}\int_{\partial\Omega}e^{2\alpha t}\Big((c_0+\nu+2|u^\infty_h|)w_{ht}^2\notag\\
	&+\frac{2}{3c_0}w_h^2w_{ht}^2\Big) d\Gamma\leq 2e^{2\alpha t}\Big(B\big(u^\infty_h,w_{ht},\Delta_h w_{ht}\big)+B\big(w_{ht},u^\infty_h,\Delta_h w_{ht}\big)\Big)\notag\\
	&\hspace{3.5cm}+2e^{2\alpha t}B\big(w_{ht},w_h,\Delta_h w_{ht}\big)+2e^{2\alpha t}B\big(w_h;w_{ht},\Delta_h w_{ht}\big)\notag\\
	&\hspace{3.5cm}+C\int_{\partial\Omega}e^{2\alpha t}\Big(w_{ht}^2+w_hw_{ht}^3+w_h^2w_{ht}^2\Big)\;d\Gamma\label{feq4.25}.
	\end{align}
	Therefore in a similar fashion as in continuous case, from \eqref{feq4.25}, we arrive at
	\begin{align*}
	\frac{d}{dt}\Big(\norm{e^{\alpha t}\nabla w_{ht}}^2&+\int_{\partial\Omega}(c_0+\nu+2|u^\infty_h|)e^{2\alpha t}w_{ht}^2\;d\Gamma+\frac{2}{3c_0}\int_{\partial\Omega}e^{2\alpha t}w_h^2w_{ht}^2\;d\Gamma\Big)+\nu\norm{e^{\alpha t}\Delta_h w_{ht}}^2\\
	&\leq Ce^{2\alpha t}\norm{w_{ht}}^2\Big(\norm{w_h}^2_2+\norm{w_h}^2\norm{\nabla w_h}^4\Big)+Ce^{2\alpha t}\norm{\nabla w_{ht}}^2\Big(1+\norm{w_h}^2+\norm{\Delta_h w_h}^2\Big)\\
	&\qquad+\int_{\partial\Omega}e^{2\alpha t}\Big(w_{ht}^2+w_h^2w_{ht}^2\Big)\;d\Gamma+Ce^{2\alpha t}\Big(\norm{w_{ht}}^4+\norm{\nabla w_{ht}}^4\Big).
	\end{align*}
	Integrate the above inequality from $0$ to $t$ and then apply the Gr\"onwall's inequality with Lemmas \ref{flm3.1}-\ref{flm3.3}, \ref{flm3.4} to obtain
	\begin{align*}
	&\Big(\norm{e^{\alpha t}\nabla w_{ht}(t)}^2+\int_{\partial\Omega}(c_0+\nu+2|u^\infty_h|)e^{2\alpha t}w_{ht}(t)^2\;d\Gamma+\frac{2}{3c_0}\int_{\partial\Omega}e^{2\alpha t}w_h(t)^2w_{ht}(t)^2\;ds\Big)\\
	&\qquad+\nu \int_{0}^{t}\norm{e^{\alpha s}\Delta_h w_{ht}(s)}^2 ds \leq C\Big(\norm{\nabla w_{ht}(0)}^2+\norm{w_{ht}(0)}_{L^2(\partial\Omega)}+\norm{w_h(0)w_{ht}(0)}^2_{L^2(\partial\Omega)}\Big)\\
	&\hspace{5.5cm}\exp\Big(C\int_{0}^{t}\big(\norm{w_h(s)}^2+\norm{\Delta_h w_h(s)}^2+\norm{\nabla w_{ht}(s)}^2\big)\;ds\Big).
	\end{align*}
	As in continuous case, we can find the value $\norm{\nabla w_{ht}(0)}$. The other two terms namely $\norm{w_{ht}(0)}_{L^2(\partial\Omega)}$ and $\norm{w_h(0)w_{ht}(0)}^2_{L^2(\partial\Omega)}$ are bounded respectively
	by $C\norm{w_{ht}(0)}^2_1$ and $C\norm{w(0)}^2_{L^4(\Omega)}\norm{w_{ht}(0)}^2_1$.
	Again, a use of Lemmas \ref{flm3.1}, \ref{flm3.2}, and \ref{flm3.4} to the above inequality completes the rest of the proof.
\end{proof}
\subsection{Error estimates}
Define an auxiliary projection $\tilde w_h\in V_h$ of $w$ through the following form
\begin{equation}\label{feq3.1}
\Big(\nabla(w-\tilde w_h),\nabla\chi\Big)+\lambda\Big(w-\tilde w_h,\chi\Big)=0\quad \forall~ \chi\in V_h,
\end{equation}
where $\lambda\geq 1$ is some fixed positive number. For a given $w,$ the existence of a unique $\tilde w_h$ follows by the Lax-Milgram lemma. 
Let $\eta:=w-\tilde w_h$ be the error involved in the auxiliary projection. Then, the following error estimates hold:
\begin{align}\label{feq3.2}
\norm{\eta}_j&\leq C h^{\min(2,m)-j}\;\norm{w}_m, \;\text{and}\notag\\
\norm{\eta _t}_j&\leq C h^{\min(2,m)-j}\norm{w_t}_m,\;\;
j=0,1 \;\mbox{ and }\; m=1,2.
\end{align}
For a proof, we refer to Thom\'{e}e \cite{thomee}.
Following Lemma \ref{x1} is needed to establish  error estimates.
\begin{lemma}\label{x1}
	Let $F\in H^{3/2+\epsilon}(\Omega)$, for some $\epsilon>0$, and $G\in H^{1/2}(\partial\Omega).$ Then $FG\in H^{1/2}(\partial\Omega)$ and
	$$\norm{FG}_{H^{1/2}(\partial\Omega)}\leq C\norm{F}_{H^{3/2+\epsilon}(\Omega)}\norm{G}_{H^{1/2}(\partial\Omega)}.$$
\end{lemma}
\begin{proof}
	For a proof see \cite{dt73}.
\end{proof}
In addition, for proving error estimates for state variable and feedback controllers, we need the following estimate of $\eta$ and $\eta_t$ at boundary which are proved in \cite{skakp2}.
\begin{lemma}\label{flm3.6}
	For smooth $\partial\Omega,$ there holds	
	\begin{align*}
	&\norm{\eta}_{L^2(\partial\Omega)}\leq Ch^{3/2}\norm{w}_2, \quad	\norm{\eta}_{H^{-1/2}(\partial\Omega)}\leq Ch^2\norm{w}_2, \quad \norm{\eta_t}_{H^{-1/2}(\partial\Omega)}\leq Ch^2\norm{w_t}_2,\\
	&\norm{\eta}_{L^4(\partial\Omega)}\leq Ch\norm{w}_2, \quad\text{and} \quad \norm{\eta_t}_{L^4(\partial\Omega)}\leq Ch\norm{w_t}_{2}.
	\end{align*}
\end{lemma}
\begin{proof}
	for a proof see \cite{skakp2}.
\end{proof}
With $e:=w-w_h,$ decompose $e:=(w-\tilde w_h)- (w_h-\tilde w_h)=:\eta-\theta,$ where $\eta=w-\tilde w_h$ and $\theta=w_h-\tilde w_h$.\\
Since estimates of $\eta$ are known from \eqref{feq3.2} and Lemma \ref{flm3.6}, it is sufficient to estimate $\theta$.
Subtracting the weak formulation \eqref{feq1.10} from \eqref{feqn3.1}, a use of \eqref{feq3.1} yields
\begin{align}\label{feqx3.1}
(\theta_t,\chi)+\nu(\nabla \theta, \nabla \chi)&=(\eta_t,\chi)-\nu\lambda(\eta,\chi)
+\Bigg(\Big(u^\infty(\nabla w\cdot {\bf{1}})+w(\nabla u^\infty\cdot {\bf{1}})\Big)-
\Big(u^\infty_h(\nabla w_h\cdot {\bf{1}})\notag\\
&\quad+w_h(\nabla u^\infty_h\cdot {\bf{1}})\Big),\chi\Bigg)+\Big((\eta-\theta)\nabla w\cdot {\bf{1}}+w_h(\nabla\eta-\nabla\theta)\cdot {\bf{1}},\chi\Big)\notag\\
&\quad+\int_{\partial\Omega}\Big((c_0+\nu+2|u^\infty|)(\eta-\theta)\chi+2(|u^\infty|-|u^\infty_h|)w_h\chi+\frac{2}{9c_0}
(w^3-w_h^3)\chi \Big) \;d\Gamma.
\end{align}
Further, \eqref{feqx3.1} can be rewritten as
\begin{align}\label{feqn3.5}
(\theta_t,\chi)+\nu(\nabla \theta, \nabla \chi)&=(\eta_t,\chi)-\nu\lambda( \eta,\chi) +\int_{\partial\Omega}\Big((c_0+\nu+2|u^\infty|)(\eta-\theta)+2(|u^\infty|-|u^\infty_h|)w_h\notag\\
&\qquad+\frac{2}{9c_0}\big(\eta^3-\theta^3+3w\eta(w-\eta)-3\theta w_h(w_h-\theta)\big)\Big) \chi \;d\Gamma\notag\\
&\qquad+\big((\eta-\theta)(\nabla w\cdot{\bf{1}})+w_h((\nabla \eta-\nabla \theta)\cdot{\bf{1}}),\chi\big)\notag\\
&\qquad+\Bigg(\Big(u^\infty(\nabla(\eta-\theta)\cdot{\bf{1}})+(u^\infty-u^\infty_h)(\nabla w_h\cdot{\bf{1}})+(\eta-\theta)(\nabla u^\infty\cdot{\bf{1}})\notag\\
&\qquad +w_h(\nabla(u^\infty-u^\infty_h)\cdot{\bf{1}})\Big),\chi\Bigg)
 .
\end{align}
Before proving the main error estimate theorem, we first estimate for $\norm{u^\infty-u^\infty_h}_{j}$ for $j=0,\hspace{0.1cm} 1$.
For steady state error estimates, subtracting the corresponding
steady state weak formulation to obtain
\begin{align}
\nu(\nabla(u^\infty-u^\infty_h),\nabla \chi )+\Big(u^\infty(\nabla u^\infty\cdot{\bf{1}})-u^\infty_h(\nabla u^\infty_h\cdot{\bf{1}}),\chi\Big)=0 \quad \forall \quad\chi\in V_h.
\end{align}
Let $\tilde u^\infty_h\in V_h$ be the elliptic projection of $u^\infty\in H^1(\Omega)$ defined by
\begin{align}
(\nabla(u^\infty-\tilde u^\infty_h))+\lambda(u^\infty-\tilde u^\infty_h,\chi)=0,
\end{align}
where $\lambda\geq 1$ is some fixed positive number.
and $\tilde u^\infty_h$ coincides with $u^\infty_h$ on the boundary.
The steady state error is splitted into two parts $\tilde e:=u^\infty-u^\infty_h=(u^\infty-\tilde u^\infty_h)-(u^\infty_h-\tilde u^\infty_h)=:\tilde\eta-\tilde\theta$, where $\tilde{\eta}$ satisfies
\begin{align*}
\norm{\tilde{\eta}}_j\leq Ch^{2-j}\norm{u^\infty}_2\quad \text{for} \quad j=0,\hspace{0.1cm}1.
\end{align*}
Now equation in $\tilde\theta$ becomes 
\begin{align}\label{fx3.5}
\nu(\nabla\tilde\theta, \nabla\chi)=\nu\lambda(\tilde \eta,\chi)+(u^\infty(\nabla(\tilde \eta-\tilde{\theta})\cdot{\bf {1}})+(\tilde{\eta}-\tilde{\theta})(\nabla u^\infty_h\cdot{\bf {1}}),\chi).
\end{align}
The following steady state error estimates hold
\begin{theorem}\label{thmx1}
For $u^\infty\in H^2$, there holds	
\begin{align*}
	\norm{u^\infty-u^\infty_h}_j\leq Ch^{2-j}\quad \text{for}\quad j=0,\hspace{0.1cm} 1.
\end{align*}
\end{theorem}
\begin{proof}
Setting $\chi=\tilde{\theta}$ in \eqref{fx3.5}, it follows that
\begin{align*}
\nu \norm{\nabla\tilde{\theta}}^2&=\nu\lambda(\nabla\tilde \eta,\nabla\tilde\theta)+(u^\infty(\nabla(\tilde \eta-\tilde{\theta})\cdot{\bf {1}})+(\tilde{\eta}-\tilde{\theta})(\nabla u^\infty_h\cdot{\bf {1}}),\tilde\theta)\\
&\leq \nu\lambda\norm{\nabla\tilde\eta}\norm{{\nabla\tilde\theta}}+N\norm{ \nabla u^\infty}\norm{\nabla\tilde\eta}\norm{\nabla\tilde\theta}+N\norm{ \nabla u^\infty}\norm{\nabla\tilde\theta}^2+\\
&\quad+N\norm{\nabla\tilde\eta}\norm{ \nabla u^\infty_h}\norm{\nabla\tilde\theta}+N\norm{ \nabla u^\infty_h}\norm{\nabla \tilde{\theta}}^2\\
&\leq \frac{\nu}{4}\norm{\nabla\tilde{\theta}}^2+C\norm{\nabla\tilde{\eta}}^2\Big(1+\norm{\nabla u^\infty}^2+\norm{ \nabla u^\infty_h}^2\Big)+N\Big(\norm{\nabla u^\infty}+\norm{ \nabla u^\infty_h}\Big)\norm{\nabla\tilde{\theta}}^2.
\end{align*}
Now using the bound of $\norm{\nabla u^\infty}$, we obtain $N\Big(\norm{ \nabla u^\infty}+\norm{\nabla u^\infty_h}\Big)\leq \frac{\nu}{2}$. Hence we arrive at
\begin{align*}
\norm{\nabla\tilde{\theta}}^2\leq C\norm{\nabla \tilde{\eta}}^2\leq Ch^2\norm{u^\infty}^2_2.
\end{align*}
Using estimate of $\norm{\nabla\tilde{\eta}}$ and a use of triangle inequality completes the first part of the proof.
For $L^2$-error estimate, we  use Aubin-Nitsche technique. Consider the problem
\begin{align}
-\nu\Delta\phi^\infty-u^\infty (\nabla\phi^\infty\cdot{\bf {1}})=\tilde e \quad \text{with} -\nu \frac{\partial \phi^\infty}{\partial n}=0,
\end{align}
where $\phi^\infty$ satisfies $\norm{\phi^\infty}_2\leq C\norm{\tilde e}$.
Now for proving error estimates, it is enough to estimate for
$\norm{\theta(t)}$.
In its weak formulation, it satisfies
\begin{align}\label{fx3.1}
\nu(\nabla\phi^\infty,\nabla v)+(u^\infty(\nabla v\cdot{\bf {1}}),\phi^\infty)+(\phi^\infty(\nabla u^\infty\cdot{\bf {1}}),v)=(\tilde e,v)\quad v\in H^1_0.
\end{align}
Set $v=\tilde e$ to obtain
\begin{align}\label{fx3.2}
\norm{\tilde e}^2=\nu(\nabla\phi^\infty,\nabla\tilde e)+(u^\infty(\nabla\tilde e\cdot{\bf {1}}),\phi^\infty)+(\phi^\infty(\nabla u^\infty\cdot{\bf {1}}),\tilde e).
\end{align}
Also subtracting the corresponding steady state weak formulation we obtain
\begin{align}\label{fx3.3}
\nu(\nabla \tilde e,\nabla \chi)+\Big(u^\infty(\nabla u^\infty\cdot{\bf {1}})-u^\infty_h(\nabla u^\infty_h\cdot{\bf {1}}),\chi\Big)=0.
\end{align}
Hence from \eqref{fx3.2} and \eqref{fx3.3}, we arrive at
\begin{align*}
\norm{\tilde e}^2&=\nu(\nabla \tilde e,\nabla (\phi^\infty-\chi))+(u^\infty(\nabla \tilde e\cdot{\bf {1}}),\phi^\infty)+(\phi^\infty(\nabla u^\infty\cdot{\bf {1}}),\tilde e)\\
&\quad-\Big(u^\infty(\nabla u^\infty\cdot{\bf {1}})-u^\infty_h(\nabla u^\infty_h\cdot{\bf {1}}),\chi\Big).
\end{align*}
Set $\chi=\tilde{\phi^\infty_h}$ satisfying $\norm{\nabla\phi^\infty-\nabla\tilde{\phi^\infty_h}}\leq Ch\norm{\phi^\infty}_2\leq Ch\norm{\tilde e}$. 
Therefore
\begin{align*}
\norm{\tilde e}^2&=\nu(\nabla \tilde e,\nabla (\phi^\infty-\tilde\phi^\infty_h))+(u^\infty(\nabla \tilde e\cdot{\bf {1}}),\phi^\infty-\tilde\phi^\infty_h)+(\tilde e(\nabla \tilde e\cdot{\bf {1}}),\phi^\infty)+(\tilde e(\nabla u^\infty_h\cdot{\bf{1}}),\phi^\infty-\tilde\phi^\infty_h)\\
&\leq Ch\norm{\nabla \tilde e}\norm{\tilde e}+Ch\norm{ \nabla u^\infty}\norm{\nabla \tilde e}\norm{\tilde e}+C\norm{\nabla \tilde e}^2\norm{\tilde e}+Ch\norm{\nabla \tilde e}\norm{\nabla u^\infty_h}\norm{\tilde e}.
\end{align*}
Hence
\begin{equation*}
\norm{\tilde e}=Ch\norm{\nabla\tilde e}+C\norm{\nabla\tilde e}^2.
\end{equation*}
\end{proof}
\begin{lemma}\label{x2}
\begin{equation*}
\norm{\tilde e}_{H^{-1/2}(\partial\Omega)}\leq Ch^2\norm{u^\infty}_2.
\end{equation*}
\end{lemma}
\begin{proof}
Proof follows similarly as to show estimate $\norm{\eta}_{H^{-1/2}(\partial\Omega)}$
in Lemma \ref{flm3.6}. Since $\tilde u^\infty_h=u^\infty_h$ on the boundary, so it completes the proof.
\end{proof}
\begin{theorem}\label{fthm3.1}
Under the assumptions $
\bf(A1)$ and $
\bf(A2)$, there holds
	\begin{align*}
	\norm{\theta(t)}^2&+e^{-2\alpha t}\int_{0}^{t}e^{2\alpha s}\Big(\frac{\beta_1}{2}\norm{\nabla \theta(s)}^2+\frac{\beta_1}{2}\norm{\theta(s)}^2_{L^2(\partial\Omega)}+\frac{1}{18c_0}\norm{\theta(s)}^4_{L^4(\partial\Omega)}\Big)\;ds\\
	&+4e^{-2\alpha t}\int_{0}^{t}e^{2\alpha s}\Big(\int_{\partial\Omega}|u^\infty|\theta(s)^2\;d\Gamma\Big)\;ds \leq Ce^{C}h^4e^{-2\alpha t},
	\end{align*}
	where $\beta_1=\min\{(\frac{\nu}{2}-2\alpha C_F), (2c_0+2\nu-2\alpha C_F)\}$. 
\end{theorem}
\begin{proof}
	Setting $\chi=\theta$ in \eqref{feqn3.5}, we obtain
	\begin{align}\label{feq3.6}
	\frac{1}{2}\frac{d}{dt}\norm{\theta}^2+\nu\norm{\nabla\theta}^2
	&=\Big((\eta_t,\theta)-\nu\lambda(\eta,\theta)\Big)+\int_{\partial\Omega}\Bigg((c_0+\nu+2|u^\infty|)(\eta-\theta)+2(|u^\infty|-|u^\infty_h|)w_h\notag\\
	&\quad+\frac{2}{9c_0}\Big(\eta^3-\theta^3+3w\eta(w-\eta)-3\theta w_h(w_h-\theta)\Big)\Bigg)\theta\; d\Gamma\notag\\
	&\quad+\Big((\eta-\theta)(\nabla w\cdot{\bf{1}})+w_h((\nabla \eta-\nabla \theta)\cdot{\bf{1}}),\theta\Big)\notag\\
	&\quad+\Bigg(\Big(u^\infty((\nabla \eta-\nabla \theta)\cdot{\bf{1}})+(u^\infty-u^\infty_h)(\nabla w_h\cdot{\bf{1}})+(\eta-\theta)(\nabla u^\infty\cdot{\bf{1}})\notag\\
	&\quad+w_h(\nabla(u^\infty-u^\infty_h)\cdot{\bf{1}})\Big),\theta\Bigg)=\sum_{i=1}^{4}I_i(\theta).
	\end{align}
	The first term $I_1(\theta)$ on the right hand side of \eqref{feq3.6} is bounded by $$I_1(\theta)=(\eta_t,\theta)-\nu\lambda(\eta,\theta)\leq C\big(\norm{\eta}^2+\norm{\eta_t}^2\big)+\frac{\epsilon}{7}\norm{\theta}^2,$$
	where the positive number $\epsilon>0$ is to be chosen later.
	The first subterm in the second term $I_2(\theta)$ on the right hand side is bounded by
	\begin{align*}
	\int_{\partial\Omega}&\Big((c_0+\nu+2|u^\infty|)(\eta-\theta)\theta+2(|u^\infty|-|u^\infty_h|)w_h\theta\Big)
	\;d\Gamma\\
	&\leq -\int_{\partial\Omega}(c_0+\nu+2|u^\infty|)\theta^2 \;d\Gamma+C\norm{\eta}_{H^{-1/2}}\Big(\norm{\theta}^2_{H^{1/2}(\partial\Omega)}+\norm{u^\infty}\norm{\theta}^2_{H^{1/2}(\partial\Omega)}\Big)\\
	&\quad+2\norm{\tilde \eta}_{H^{-1/2}(\partial\Omega)}\norm{w_h\theta}_{H^{1/2}(\partial\Omega)}\\
	&\leq -\int_{\partial\Omega}(c_0+\nu+2|u^\infty|)\theta^2 \;d\Gamma+\frac{\epsilon}{7}\norm{\theta}^2+\frac{\nu}{24}\norm{\nabla \theta}^2+C\norm{\eta}^2_{H^{-1/2}}+Ch^4\norm{u^\infty}_2^2\norm{w_h}^2_2.
	\end{align*}
	For other subterms in the second term $I_2(\theta),$ we note that
	\begin{align*}
	\frac{2}{9c_0}\int_{\partial\Omega}\eta^3\theta d\Gamma\leq \frac{2}{9c_0}\norm{\eta}^3_{L^4(\partial\Omega)}\norm{\theta}_{L^4(\partial\Omega)}\leq \frac{1}{9c_0}\frac{1}{4}\norm{\theta}^4_{L^4(\partial\Omega)}+\frac{C}{c_0}\norm{\eta}^4_{L^4(\partial\Omega)},
	\end{align*}
	\begin{align*}
	\frac{2}{9c_0}\int_{\partial\Omega}3w^2\eta\theta d\Gamma
	&\leq C\norm{\tilde \eta}_{H^{-1/2}(\partial\Omega)}\norm{w^2\theta}_{H^{1/2}(\partial\Omega)}
	\\
	&\leq \frac{\epsilon}{7}\norm{\theta}^2+\frac{\nu}{24}\norm{\nabla \theta}^2+C\norm{w}^4_2\norm{\eta}^2_{H^{-1/2}(\partial\Omega)},
	\end{align*}
	\begin{align*}
	\frac{2}{9c_0}\int_{\partial\Omega}3w\eta^2\theta d\Gamma\leq \frac{\epsilon}{7}\norm{\theta}^2+\frac{\nu}{24}\norm{\nabla \theta}^2+C\norm{w}^2_{L^4(\partial\Omega)}\norm{\eta}^4_{L^4(\partial\Omega)},
	\end{align*}
	and
	\begin{align*}
	\frac{2}{9c_0}\int_{\partial\Omega}3w_h\theta^3 d\Gamma\leq\frac{6}{9c_0}\Bigg(\int_{\partial\Omega} w_h^2\theta^2 d\Gamma+\frac{1}{4}\int_{\partial\Omega}\theta^4 d\Gamma\bigg).
	\end{align*}
	For the third term $I_3(\theta)$ first we bound the following sub-terms as
	\begin{align*}
	\big((\eta-\theta)&(\nabla w\cdot{\bf{1}}),\theta\big)-(w_h(\nabla \theta\cdot{\bf{1}}),\theta)\\
	&\leq C\norm{\eta}\norm{\nabla w}_{L^4}\norm{\theta}_{L^4}+C\norm{\theta}^2_{L^4}\norm{\nabla w}\\
	&\qquad+C\norm{ w_h}_{L^4}\norm{\theta}_{L^4}\norm{\nabla \theta}
	\\
	&\leq \frac{\nu}{24}\norm{\nabla \theta}^2+\frac{\epsilon}{7}\norm{\theta}^2+C\norm{\theta}^2\Big(\norm{w}^2_2+\norm{w_h}^2+\norm{\nabla w_h}^2\\
	&\qquad+\norm{w_h}^2\norm{\nabla w_h}^2\Big)+C\norm{\eta}^2\Big(1+\norm{w}^2+\norm{\Delta w}^2\Big).
	\end{align*}
	The other sub-term in $I_3(\theta)$ is bounded by
	\begin{align*}
	\big(w_h(\nabla\eta\cdot{\bf{1}}),\theta\big)&=-\big(w_h(\nabla\theta\cdot{\bf{1}}),\eta\big)-\big(\eta (\nabla w_h\cdot{\bf{1}}),\theta\big)+\sum_{i=1}^{2}\int_{\partial\Omega}w_h\eta\nu_i\theta\; d\Gamma\\
	&\leq C\norm{\eta}\norm{\nabla \theta}\norm{w_h}_{L^\infty}+C\norm{\eta}\norm{\nabla w_h}_{L^4}\norm{\theta}_{L^4}\\
	&\quad+C\norm{\eta}_{H^{-1/2}(\partial\Omega)}\norm{w_h\theta}_{H^{1/2}(\partial\Omega)}\\
	&\leq \frac{\nu}{24}\norm{\nabla \theta}^2+\frac{\epsilon}{7}\norm{\theta}^2+\frac{c_0}{4}\norm{\theta}^2_{L^2(\partial\Omega)}+C\norm{\eta}^2\Big(1+\norm{w_h}^2+\norm{\Delta_hw_h}^2\Big)\\
	&\qquad+C\norm{\theta}^2\Big((\norm{w_h}^2_1+\norm{\Delta_hw_h}^2)\Big)+C\norm{\eta}^2_{H^{-1/2}(\partial\Omega)}.
	\end{align*}
	For the fourth term $I_4(\theta)$, bound the subterm as in third term
	\begin{align*}
(u^\infty(\nabla \eta\cdot{\bf{1}}),\theta)&=-(u^\infty(\nabla\theta\cdot{\bf{1}}),\eta)-(\eta(\nabla u^\infty\cdot{\bf{1}}),\theta)+\sum_{i=1}^{2}\int_{\partial\Omega}u^\infty\eta n_i\theta\;d\Gamma\\
&\leq \norm{u^\infty}_{L^\infty}\norm{\nabla\theta}\norm{\eta}+\norm{\eta}\norm{\nabla u^\infty}_{L^4}\norm{\theta}_{L^4}+C\norm{\eta}_{H^{-1/2}(\partial\Omega)}\norm{u^\infty\theta}_{H^{1/2}(\partial\Omega)},
	\end{align*}
	\begin{align*}
\Big(-u^\infty(\nabla\theta\cdot{\bf{1}})+(\eta-\theta)(\nabla u^\infty\cdot{\bf{1}}),\theta\Big)&\leq 2N\norm{ \nabla u^\infty}\norm{\nabla\theta}^2+\norm{\eta}\norm{\nabla u^\infty}_{L^4}\norm{\theta}_{L^4}\\
&\leq \frac{\nu}{2}\norm{\nabla\theta}^2+\norm{\eta}\norm{\nabla u^\infty}_{L^4}\norm{\theta}_{L^4},
	\end{align*}
	\begin{align*}
	\Big((u^\infty-u^\infty_h)(\nabla w_h\cdot{\bf{1}}),\theta\Big)\leq \norm{u^\infty-u^\infty_h}
	\norm{\nabla w_h}_{L^4}\norm{\theta}_{L^4},
	\end{align*}
	\begin{align*}
	\Big(w_h(\nabla(u^\infty-u^\infty_h)\cdot{\bf{1}}),\theta\Big)&=-\Big((u^\infty-u^\infty_h)(\nabla w_h\cdot{\bf{1}}),\theta\Big)-\Big(w_h(\nabla\theta\cdot{\bf{1}}),(u^\infty-u^\infty_h)\Big)\\
	&\quad+\sum_{i=1}^{2}\int_{\partial\Omega}w_h(u^\infty-u^\infty_h)n_i\theta \;d\Gamma\\
	&\leq \norm{u^\infty-u^\infty_h}\norm{\nabla w_h}_{L^4}\norm{\theta}_{L^4}+\norm{w_h}_{L^\infty}\norm{\nabla\theta}\norm{u^\infty-u^\infty_h}\\
	&\quad+C\norm{u^\infty-u^\infty_h}_{H^{-1/2}(\partial\Omega)}\norm{w_h\theta}_{H^{1/2}(\partial\Omega)}.
	\end{align*}
	Now, contribution from the fourth term $I_4(\theta)$ becomes 
	\begin{align*}
I_4(\theta)&\leq \frac{\epsilon}{7}\norm{\theta}^2+\frac{\nu}{24}\norm{\nabla\theta}^2
+\frac{\nu}{2} \norm{\nabla\theta}^2+C\Big(\norm{\eta}^2+\norm{\eta}^2_{H^{-1/2}(\partial\Omega)}\Big)\\
&\qquad+Ch^4\norm{u^\infty}^2_2(\norm{w_h}^2_1+\norm{\Delta_hw_h}^2).
	\end{align*}
	Finally, using Lemmas \ref{flm1}-\ref{flm4}, \ref{flm3.1}-\ref{flm3.2} and \ref{flm3.6}, we arrive from \eqref{feq3.6} at
	\begin{align}
	\frac{d}{dt}\norm{\theta}^2&+\frac{\nu}{2}\norm{\nabla \theta}^2+2\int_{\partial\Omega}(c_0+\nu+2|u^\infty|)\theta^2\;d\Gamma+\frac{1}{18c_0}\norm{\theta}^4_{L^4(\partial\Omega)}\notag\\
	&\leq \epsilon\norm{\theta}^2+C\norm{\eta}^2(1+\norm{w}^2+\norm{\Delta w}^2+\norm{\nabla w_h}^2+\norm{\Delta_h w_h}^2)+C\norm{\theta}^2\Big(\norm{w}^2_2
	+\norm{w_h}^2_1\notag\\
	&\qquad+\norm{\Delta_hw_h}^2\Big)+C\norm{\eta}^2_{H^{-1/2}(\partial\Omega)}+Ch^4(\norm{w_h}^2_1+\norm{\Delta_hw_h}^2)\norm{u^\infty}^2_2+C\norm{\eta}^4_{L^4(\partial\Omega)}\label{feq3.10}.
	\end{align}
	Multiply \eqref{feq3.10} by $e^{2\alpha t}$. Using the Friedrichs's inequality
	$$-2\alpha e^{2\alpha t}\norm{\theta}^2\geq -2\alpha C_Fe^{2\alpha t}\big(\norm{\nabla\theta}^2+\norm{\theta}^2_{L^2(\partial\Omega)}\big),$$ and  Lemmas \ref{flm1}-\ref{flm2}, \ref{flm3.2}, \ref{flm3.4} and \ref{flm3.6} in \eqref{feq3.10}, it follows choosing $\epsilon=\frac{\beta_1}{2C_F}$ that
	\begin{align*}
	\frac{d}{dt}\big(\norm{e^{\alpha t}\theta}^2\big)&+e^{2\alpha t}\Bigg(\Big(\frac
{\nu}{2}-2\alpha C_F\Big)\norm{\nabla \theta}^2+\Big(2c_0+2\nu-2\alpha C_F\Big)\norm{\theta}^2_{L^2(\partial\Omega)}\\
	&+4\int_{\partial\Omega}|u^\infty|\theta^2\;d\Gamma+\frac{1}{18c_0}\norm{\theta}^4_{L^4(\partial\Omega)}\Bigg)\\
	&\leq Ce^{2\alpha t}\Big(\norm{\eta}^2+\norm{\eta_t}^2\Big)+Ce^{2\alpha t}\norm{\theta}^2\Big(\norm{w}^2_2+(\norm{w_h}^2_1+\norm{\Delta_hw_h}^2)\Big)\\
	&\qquad+Ce^{2\alpha t}\Big(\norm{\eta}^2_{H^{-1/2}(\partial\Omega)}+C\norm{\eta}^4_{L^4(\partial\Omega)}\Big)+Ch^4\norm{u^\infty}^2_2(\norm{w_h}^2_1+\norm{\Delta_hw_h}^2)\\
	&\qquad+\frac{\beta_1}{2}\Big(\norm{ \nabla\theta}^2+\norm{\theta}^2_{L^2(\partial\Omega)}\Big).
	\end{align*}
	Integrate the above inequality from $0$ to $t$ and then use the Gr\"onwall's inequality to obtain
	\begin{align*}
	\norm{e^{\alpha t}\theta(t)}^2&+\int_{0}^{t}e^{2\alpha s}\Big(\frac{\beta_1}{2}\norm{\nabla \theta(s)}^2+\frac{\beta_1}{2}\norm{\theta(s)}^2_{L^2(\partial\Omega)}+\frac{1}{18c_0}\norm{\theta(s)}^4_{L^4(\partial\Omega)}\Big)\;ds\\
	&\quad+4\int_{0}^{t}e^{2\alpha s}\Big(\int_{\partial\Omega}|u^\infty|\theta(s)^2\;d\Gamma\Big)\;ds\\
	&\leq Ch^4\Big(\int_{0}^{t}\big(\norm{w(s)}^2_2+\norm{w_t(s)}^2_2\big)\;ds\Big)\exp\Bigg(\int_{0}^{t}\Big(\norm{w(s)}^2_2+(\norm{w_h(s)}^2_1+\norm{\Delta_hw_h(s)}^2)\Big)\;ds\Bigg).
	\end{align*}
	A use of Lemmas \ref{flm1}-\ref{flm5} and \ref{flm3.1}-\ref{flm3.2} to the above inequality with a multiplication of $e^{-2\alpha t}$ completes the  the proof.
\end{proof}
\begin{theorem}\label{fthm3.2}
	Under the assumptions $
	\bf(A1)$ and $
	\bf(A2)$, there holds
	\begin{align*}
	\nu \norm{\nabla \theta(t)}^2+4&\int_{\partial\Omega}(c_0+\nu+2|u^\infty|)\theta(t)^2\;d\Gamma+\frac{2}{9c_0}\norm{\theta(t)}^2_{L^4(\partial\Omega)}+\frac{4}{3c_0}\int_{\partial\Omega}w_h(t)^2\theta(t)^2\;d\Gamma\\
	&\quad+e^{-2\alpha t}\int_{0}^{t}e^{2\alpha s}\norm{\theta_t(s)}^2\;ds\leq Ce^{C}h^4e^{-2\alpha t}.
	\end{align*}
\end{theorem}
\begin{proof}
	Set $\chi=\theta_t$ in \eqref{feqn3.5} to obtain
	\begin{align}
	\norm{\theta_t}^2+\frac{\nu}{2}\frac{d}{dt}\norm{\nabla\theta}^2
	&=\sum_{i=1}^{4}I_i(\theta_t)\label{feq5.1}.
	\end{align}
	The first term $I_1(\theta_t)$ on the right hand side of \eqref{feqn3.5} is bounded by
	\begin{equation*}
	I_1(\theta_t)=\Big((\eta_t,\theta_t)-\lambda\nu(\eta,\theta_t)\Big)\leq \frac{1}{6}\norm{\theta_t}^2+C\Big(\norm{\eta}^2+\norm{\eta_t}^2\Big).
	\end{equation*}
For the second term $I_2(\theta_t)$ on the right hand side of \eqref{feqn3.5}, first bound the subterms as
	\begin{align*}
\Big<&(c_0+\nu+2|u^\infty|)\eta-\theta,\theta_t\Big>_{\partial\Omega}\\
&=-\frac{1}{2}\frac{d}{dt}((c_0+\nu+2|u^\infty|)\theta^2\;d\Gamma)+\frac{d}{dt}\Big(\int_{\partial\Omega}(c_0+\nu+2|u^\infty|)\eta\theta\;d\Gamma\Big)-\int_{\partial\Omega}(c_0+\nu+2|u^\infty|)\eta_t\theta\;d\Gamma\\
	&\leq -\frac{1}{2}\frac{d}{dt}((c_0+\nu+2|u^\infty|)\theta^2\;d\Gamma)+\frac{d}{dt}\Big(\int_{\partial\Omega}(c_0+\nu+2|u^\infty|)\eta\theta\;d\Gamma\Big)\\
	&\qquad+C\Big(\norm{\eta_t}_{H^{-1/2}(\partial\Omega)}+\norm{\theta}^2+\norm{\nabla\theta}^2\Big).
	\end{align*}
	For the other subterms in $I_2(\theta_t)$ on the right hand side of \eqref{feqn3.5}, the bounds are as follows
	\begin{align*}
	\frac{2}{9c_0}\int_{\partial\Omega}\eta^3\theta_t \;d\Gamma
	&\leq \frac{2}{9c_0}\frac{d}{dt}\Big(\int_{\partial\Omega}\eta^3\theta \;d\Gamma\Big)+C\Big(\norm{\eta}^4_{L^4(\partial\Omega)}+\norm{\eta_t}^2_{L^4(\partial\Omega)}\norm{\theta}^2_1\Big),
	\end{align*}
	\begin{align*}
	\frac{2}{9c_0}\int_{\partial\Omega}3w^2\eta\theta_t\;d\Gamma&\leq\frac{2}{3c_0}\frac{d}{dt}\Big(\int_{\partial\Omega}w^2\eta \theta\;d\Gamma\Big)+C\norm{w}^2_2\big(\norm{\theta}^2+\norm{\nabla \theta}^2\big)\\
	&\qquad+C\norm{\eta_t}^2_{H^{-1/2}(\partial\Omega)}+C\norm{\eta}^2_{H^{-1/2}(\partial\Omega)}\norm{w_t}^2_2,
	\end{align*}
	\begin{align*}
	-\frac{2}{9c_0}3\int_{\partial\Omega}w\eta^2\theta_t\;d\Gamma&\leq-\frac{2}{3c_0}\frac{d}{dt}\Big(\int_{\partial\Omega}w\eta^2\theta\;d\Gamma\Big)+C\norm{\eta}^4_{L^4(\partial\Omega)}\\
	&\qquad+C\norm{\eta}^2_{L^4(\partial\Omega)}\norm{\eta_t}^2_{L^4(\partial\Omega)}+C\norm{\theta}^2_1\Big(\norm{w}^2_{L^4(\partial\Omega)}+\norm{w_t}^2_{L^4(\partial\Omega)}\Big),
	\end{align*}
	\begin{align*}
	-\frac{2}{3c_0}\int_{\partial\Omega}w_h^2\theta\theta_t\;d\Gamma&\leq -\frac{1}{3c_0}\frac{d}{dt}\Big(\int_{\partial\Omega}w_h^2\theta^2 \;d\Gamma\Big)+C\norm{\theta}^2_1\Big(\norm{w}^2_{L^4(\partial\Omega)}+\norm{w_{ht}}^2_{L^4(\partial\Omega)}\Big),
	\end{align*}
	and
	\begin{align*}
	\frac{2}{9c_0}\int_{\partial\Omega}3w_h\theta^2\theta_t\;d\Gamma&\leq \frac{2}{9c_0}\frac{d}{dt}\Big(\int_{\partial\Omega}w_h\theta^3\;d\Gamma\Big)+C\norm{\theta}^2_1+C\norm{\theta}^4_{L^4(\partial\Omega)}\norm{w_{ht}}^2_{L^4(\partial\Omega)}.
	\end{align*}
	For the third term $I_3(\theta_t)$ on the right hand side of \eqref{feqn3.5}, first we rewrite the sub terms as
	\begin{align*}
	(\eta(\nabla w\cdot{\bf{1}}),\theta_t)=\frac{d}{dt}\Big((\eta(\nabla w\cdot {\bf{1}}),\theta)\Big)-(\eta_t(\nabla w\cdot {\bf{1}}),\theta)-\Big(\eta(\nabla w\cdot {\bf{1}})_t,\theta\Big),
	\end{align*}
	and using integration by parts
	\begin{align*}
	(w_h(\nabla \eta\cdot{\bf{1}}),\theta_t)&=-\frac{d}{dt}\Big((\eta(\nabla\theta\cdot {\bf{1}}),w_h)+\big(\eta(\nabla w_h \cdot {\bf{1}}),\theta\big)-\sum_{i=1}^{2}\int_{\partial\Omega}\eta w_h\nu_i\theta\; d\Gamma\Big)\\
	&\qquad +\Big((\eta w_h)_t,\nabla\theta\cdot{\bf{1}}\Big)+\Big((\eta(\nabla w_h\cdot{\bf{1}}))_t,\theta\Big)-\sum_{i=1}^{2}\int_{\partial\Omega}(\eta w_h)_t\nu_i\theta \;d\Gamma,
	\end{align*}
and hence it follows that
	\begin{align*}
(\eta(\nabla w\cdot{\bf{1}}),\theta_t)&\leq\frac{d}{dt}\Big((\eta(\nabla w\cdot {\bf{1}}),\theta)\Big)+ C\norm{\eta_t}^2\Big(1+\norm{w}^2_2\Big)+C\norm{\theta}^2\Big(1+\norm{\Delta w}^2+\norm{\Delta_t}^2\Big)\\
&\qquad+C\norm{\nabla \theta}^2\Big(\norm{w}^2+\norm{w_t}^2\Big)+C\norm{\eta}^2\norm{\Delta w_t}^2.
	\end{align*}
	Similarly,
	\begin{align*}
\Big(w_h(\nabla\eta \cdot{\bf{1}}),\theta_t\Big)&\leq -\frac{d}{dt}\Big((\eta(\nabla\theta\cdot {\bf{1}}),w_h)+\big(\eta(\nabla w_h \cdot {\bf{1}}),\theta\big)-\sum_{i=1}^{2}\int_{\partial\Omega}\eta w_h\nu_i\theta\; d\Gamma\Big)\\
&\quad+ C\norm{\eta_t}^2\Big(1+\norm{w_h}^2+\norm{\Delta_hw_h}^2\Big)+C\norm{\eta}^2\Big(1+\norm{w_{ht}}^2+\norm{\Delta_h w_{ht}}^2\Big)\\
&\quad+C\norm{\theta}^2\Big(1+\norm{\Delta_hw_h}^2+\norm{\Delta_hw_{ht}}^2\Big)+C\norm{\nabla\theta}^2\\
&\quad+C\Big(\norm{\eta}^2_{H^{-1/2}(\partial\Omega)}+\norm{\eta_t}^2_{H^{-1/2}(\partial\Omega)}\Big).
	\end{align*}
	The other two sub-terms in the fourth term are bounded by
	\begin{align*}
	-\Big(\theta(\nabla w\cdot{\bf{1}}),\theta_t\Big)-\Big(w_h(\nabla\theta\cdot{\bf{1}}),\theta_t\Big)&\leq \frac{1}{6}\norm{\theta_t}^2+C\norm{\theta}^2\Big(\norm{w}^2+\norm{\Delta w}^2+(\norm{w_h}^2_1+\norm{\Delta_hw_h}^2)\Big).
	\end{align*}
For the fourth term $I_4(\theta_t)$ on the right hand side of \eqref{feqn3.5}, subterms can be rewritten as 
		\begin{align*}
	\Big(u^\infty&((\nabla \eta-\nabla\theta)\cdot {\bf{1}})+(\eta-\theta)(\nabla u^\infty \cdot {\bf{1}}),\theta_t\Big)\\
		&\leq\frac{d}{dt}\Big(-(\eta(\nabla\theta\cdot {\bf{1}}),u^\infty)-(\eta(\nabla u^\infty\cdot {\bf{1}}),\theta)+\sum_{i=1}^{2}u^\infty\eta n_i\theta\;d\Gamma\Big)+C\norm{\eta_t}\norm{\nabla u^\infty}_{L^4}\norm{\theta}_{L^4}\\
		&\quad+C\norm{u^\infty}_{L^\infty}\norm{\nabla\theta}\norm{\theta_t}+C\norm{\theta}_{L^4}\norm{\nabla u^\infty}_{L^4}\norm{\theta_t}\\
		&\quad+C\norm{u^\infty}_{L^\infty}\norm{\nabla \theta}\norm{\eta_t}+
		C\norm{\eta_t}_{H^{-1/2}(\partial\Omega)}\norm{u^\infty}_2\norm{\theta}_1,
		\end{align*}
		\begin{align*}
\Big((u^\infty-u^\infty_h)(\nabla w_h\cdot {\bf{1}}),\theta_t\Big)=\frac{d}{dt}\Big((u^\infty-u^\infty_h)(\nabla w_h\cdot {\bf{1}}),\theta\Big)-\Big((u^\infty-u^\infty_h)(\nabla w_{ht}\cdot {\bf{1}}),\theta\Big),
		\end{align*}
		\begin{align*}
\Big(w_h(\nabla(u^\infty-u^\infty_h)\cdot {\bf{1}}),\theta_t\Big)&=\frac{d}{dt}\Big((-(u^\infty-u^\infty_h)(\nabla w_h\cdot {\bf{1}}),\theta)-\big((u^\infty-u^\infty_h)(\nabla \theta\cdot {\bf{1}}),w_h\big)\\
&\quad+\sum_{i=1}^{2}\int_{\partial\Omega}w_h(u^\infty-u^\infty_h)n_i\theta d\;\Gamma\Big)+\Big(((u^\infty-u^\infty_h)(\nabla w_{ht}\cdot {\bf{1}}),\theta)\\&\qquad+(w_{ht}(\nabla \theta\cdot {\bf{1}}),(u^\infty-u^\infty_h))-\sum_{i=1}^{2}\int_{\partial\Omega}w_{ht}(u^\infty-u^\infty_h)n_i\theta\;d\Gamma\Big),
		\end{align*}
		and hence adding all the subterms
		\begin{align*}
		I_4(\theta_t)
		&\leq \frac{d}{dt}\Big(-(\eta(\nabla\theta\cdot {\bf{1}}),u^\infty)-(\eta(\nabla u^\infty\cdot {\bf{1}}),\theta)+\sum_{i=1}^{2}\int_{\partial\Omega}w_h(u^\infty-u^\infty_h)\nu_i\theta d\Gamma\\
		&\quad-(w_h(\nabla\theta\cdot {\bf{1}}),u^\infty-u^\infty_h)+\sum_{i=1}^{2}\int_{\partial\Omega}\eta\nu_i\theta d\Gamma\Big) +\frac{1}{6}\norm{\theta_t}^2+C\Big(\norm{\theta}^2+\norm{\nabla\theta}^2\Big)\\
		&\qquad+C\Big(\norm{u^\infty}^2+\norm{\Delta u^\infty}^2\Big)\norm{\eta_t}^2+C\norm{\eta_t}^2_{L^4(\partial\Omega)}+C\norm{u^\infty-u^\infty_h}^2\Big(\norm{w_{ht}}^2+\norm{\Delta_hw_{ht}}^2\Big)\\
		&\qquad+C\Big(\norm{\eta}^2_{H^{-1/2}(\partial\Omega)}+\norm{\eta_t}^2_{H^{-1/2}(\partial\Omega)}\Big)+C\norm{u^\infty-u^\infty_h}^2_{H^{-1/2}(\partial\Omega)}\norm{w_{ht}}^2_2.
		\end{align*}
Finally, from \eqref{feqn3.5}, we arrive at
	\begin{align*}
	\frac{d}{dt}\Big(\nu\norm{\nabla\theta}^2&+2\int_{\partial\Omega}(c_0+\nu+2|u^\infty|)\theta^2\;d\Gamma+\frac{1}{9c_0}\norm{\theta}^2_{L^4(\partial\Omega)}+\frac{2}{3c_0}\int_{\partial\Omega}w_h^2\theta^2\;d\Gamma\Big)+\norm{\theta_t}^2\\
	&\leq \frac{d}{dt}\Big(-\big(\eta (\nabla w_h\cdot{\bf{1}}),\theta \big)-\big(\eta (\nabla u^\infty\cdot{\bf{1}}),\theta \big)+\sum_{i=1}^{2}\int_{\partial\Omega}\eta w_h\nu_i\theta\;d\Gamma\\
	&\qquad+\int_{\partial\Omega}(c_0+\nu+2|u^\infty|)\eta\theta\;d\Gamma-\big(u^\infty(\nabla\theta\cdot{\bf{1}}),\eta\big)-\big(w_h(\nabla\theta\cdot{\bf{1}}),u^\infty-u^\infty_h\big)\\
	&\quad+\sum_{i=1}^{2}\int_{\partial\Omega}u^\infty\eta n_i\theta\;d\Gamma+\sum_{i=1}^{2}\int_{\partial\Omega}w_h(u^\infty-u^\infty_h)n_i\theta\;d\Gamma\Big)\\
	&\quad+C\norm{\eta_t}^2\Big(1+\norm{w}^2+\norm{\Delta w}^2+\norm{\Delta_hw_h}^2\Big)\\
	&\qquad+C\norm{\eta}^2\Big(1+\norm{w_{ht}}^2+\norm{\Delta_hw_{ht}}^2\Big)+C\norm{\theta}^2\Big(1+\norm{w}^2_2+(\norm{w_h}^2_1+\norm{\Delta_hw_h}^2)\\
	&\qquad+\norm{\Delta_hw_{ht}}^2\Big)+Ch^4\Big(1+\norm{w}^2_2+\norm{w_t}^2_2\Big)+Ch^2\norm{\theta}^2_1+C\norm{\eta}^2_{L^4(\partial\Omega)}{\eta_t}^2_{L^4(\partial\Omega)}.
	\end{align*}
	Multiply the above inequality by $e^{2\alpha t}$ and use Lemmas \ref{flm1}, \ref{flm4} and\ref{flm3.1}- \ref{flm3.3} to obtain
	\begin{align*}
	\frac{d}{dt}\Bigg(e^{2\alpha t}\Big(&\nu\norm{\nabla\theta}^2+2\int_{\partial\Omega}(c_0+\nu+2|u^\infty|)\theta^2\;d\Gamma+\frac{1}{9c_0}\norm{\theta}^2_{L^4(\partial\Omega)}+\frac{2}{3c_0}\int_{\partial\Omega}w_h^2\theta^2\;d\Gamma\Big)\Bigg)+e^{2\alpha t}\norm{\theta_t}^2\\
	&\leq \frac{d}{dt}\Bigg(e^{2\alpha t}\Big(-\big(\eta (\nabla w_h\cdot{\bf{1}}),\theta \big)+\sum_{i=1}^{2}\int_{\partial\Omega}\eta w_h\nu_i\theta\;d\Gamma+\int_{\partial\Omega}(c_0+\nu+2|u^\infty|)\eta\theta\;d\Gamma\\
	&\qquad-\big(u^\infty(\nabla\theta\cdot{\bf{1}}),\eta\big)-\big(\eta (\nabla u^\infty\cdot{\bf{1}}),\theta\big)-\big(w_h(\nabla\theta\cdot{\bf{1}}),u^\infty-u^\infty_h\big)\\
	&\qquad+\sum_{i=1}^{2}\int_{\partial\Omega}u^\infty\eta n_i\theta\;d\Gamma+\sum_{i=1}^{2}\int_{\partial\Omega}w_h(u^\infty-u^\infty_h)n_i\theta\;d\Gamma\Big)\Bigg)\\
	&\quad+Ch^4e^{2\alpha t}\Big(\norm{w}^2_2+\norm{w_t}^2_2+\norm{w_{ht}}^2_2+(\norm{w_h}^2_1+\norm{\Delta_hw_h}^2)\Big)+Ce^{2\alpha t}\norm{\theta}^4_{L^4(\partial\Omega)}\norm{w_{ht}}^2_1\\
	&\quad+Ce^{2\alpha t}\norm{\theta}^2\Big(1+\norm{w}^2_2+(\norm{w_h}^2_1+\norm{\Delta_hw_h}^2)+\norm{\Delta_hw_{ht}}^2\Big)+Ce^{2\alpha t}\norm{\nabla\theta}^2.
	\end{align*}
	Integrate the above inequality from $0$ to $t$. Then multiply the resulting inequality by $e^{-2\alpha t}$ and use Lemmas \ref{flm2}, \ref{flm4}- \ref{flm5} and \ref{flm3.2}-\ref{flm3.5} and Theorem \ref{fthm3.1} to arrive at
	\begin{align*}
	\Big(\nu\norm{\nabla\theta(t)}^2&+2\int_{\partial\Omega}(c_0+\nu+2|u^\infty|)\theta(t)^2\;d\Gamma+\frac{1}{9c_0}\norm{\theta(t)}^2_{L^4(\partial\Omega)}+\frac{2}{3c_0}\int_{\partial\Omega}w_h(t)^2\theta(t)^2\;d\Gamma\Big)\\
	&+e^{-2\alpha t}\int_{0}^{t}e^{2\alpha s}\norm{\theta_t(s)}^2\;ds\\
	&\leq \Big(-\big(\eta (\nabla w_h\cdot{\bf{1}}),\theta \big)+\sum_{i=1}^{2}\int_{\partial\Omega}\eta w_h\nu_i\theta\;d\Gamma+\int_{\partial\Omega}(c_0+\nu+2|u^\infty|)\eta\theta\;d\Gamma\\
	&\qquad-\big(u^\infty(\nabla\theta\cdot{\bf{1}}),\eta\big)-\big(w_h(\nabla\theta\cdot{\bf{1}}),u^\infty-u^\infty_h\big)+\sum_{i=1}^{2}\int_{\partial\Omega}u^\infty\eta n_i\theta\;d\Gamma\\
	&\qquad-\big(\eta(\nabla u^\infty\cdot{\bf{1}}),\theta\big)+\sum_{i=1}^{2}\int_{\partial\Omega}w_h(u^\infty-u^\infty_h)n_i\theta\;d\Gamma\Big)\\
	&\qquad+Ch^4e^{-2\alpha t}\Big(\norm{w_0}_3\Big)\exp\Big(C\norm{w_0}_2\Big)\\
	&\leq \frac{\nu}{2}\norm{\nabla \theta(t)}^2+Ch^4e^{-2\alpha t}\Big(\norm{w_0}_3\Big)\exp\Big(C\norm{w_0}_2\Big).
	\end{align*}
Finally, it follows that
	\begin{align*}
	\nu \norm{\nabla \theta(t)}^2+4\int_{\partial\Omega}(c_0+\nu+2|u^\infty|)\theta(t)^2\;d\Gamma+&\frac{2}{9c_0}\norm{\theta(t)}^2_{L^4(\partial\Omega)}+\frac{4}{3c_0}\int_{\partial\Omega}w_h^2\theta(t)^2\;d\Gamma\\
	&+e^{-2\alpha t}\int_{0}^{t}e^{2\alpha s}\norm{\theta_t(s)}^2\;ds\leq Ce^{C}h^4e^{-2\alpha t}.
	\end{align*}
	This completes the proof.
\end{proof}
\begin{theorem}\label{fthm3.3}
	Under the assumptions $
	\bf(A1)$ and $
	\bf(A2)$, there holds
	\begin{align*}
	\norm{w-w_h}_{L^\infty(H^i)}\leq Ce^{C}h^{2-i}e^{-\alpha t},\quad i=0,\hspace{0.1cm}1
	\end{align*}
	and
	\begin{align*}
	\norm{v_{0t}-v_{0ht}}_{L^\infty(L^2(\partial\Omega))}\leq Ce^{C}h^{3/2}e^{-\alpha t}.
	\end{align*}
\end{theorem}
\begin{proof}
From estimates \eqref{feq3.2} and Theorems \ref{fthm3.1}-\ref{fthm3.2} with a use of triangle inequality completes the first part of the proof.
	
	For the second part, we note that
	\begin{align*}
	v_{0t}-v_{0ht}=-\frac{1}{\nu}\Big((c_0+\nu+2|u^\infty|)(\eta-\theta)+\frac{2}{9c_0}(\eta-\theta)(w^2+ww_h+w_h^2)\Big).
	\end{align*}
Therefore,
	\begin{align*}
	\norm{v_{0t}-v_{0ht}}_{L^\infty(L^2(\partial\Omega))}\leq C\Big(\norm{\eta}_{L^\infty(L^2(\partial\Omega))}+\norm{\theta}_{L^\infty(L^2(\partial\Omega))}\Big)\Big(1+\norm{w}^2_{L^\infty(L^4(\partial\Omega))}+\norm{w_h}^2_{L^\infty(L^4(\partial\Omega))}\Big).
	\end{align*}
	A use of Lemmas \ref{flm2}, \ref{flm3.2} and \ref{flm3.6} and Theorem \ref{fthm3.2} completes the proof.
\end{proof}
\section{Numerical experiments}
Now in this section, our aim is to conduct some numerical experiments to show the order of convergence  for the state variable and for the feedback control law. In addition, stabilization result is also shown numerically.
For complete discrete scheme
the time variable is discretized by replacing the time derivative with difference quotient. 
Based on backward Euler method, we now discretize the semidiscrete solution.
Let $W^n$ be an approximation of $w(t)$ in $V_h$ at $t=t_n=nk$ where
$0<k<1$ denote the time step size and  
$t_n=nk,$ $n$ is nonnegative integer. For smooth function $\phi$ defined on $[0,\infty),$
set $\phi ^n=\phi(t_n)$ and $\bar{\partial}_t\phi^n=\frac{(\phi^n-\phi^{n-1})}{k}$.\\
 Now the backward Euler method applied to \eqref{feq1.7} determines a sequence of functions $\{{W^n}\}_{n\geq 1}\in V_h$ as a solution of
\begin{align}
(\bar{\partial}_tW^n,\varphi_h)&+\nu(\nabla W^n,\nabla\varphi_{h})+\Big(u^\infty(\nabla W^n\cdot{\bf {1}})+W^n(\nabla u^\infty\cdot{\bf {1}}),\varphi_h\Big)+\big(W^n(\nabla W^n\cdot{\bf {1}}),\varphi_h)\notag\\
&+\Big<(c_0+\nu+2|u^\infty|)W^n+\frac{2}{9c_0}(W^n)^3),\varphi_h\Big>=0  \quad \forall~ \varphi_h \in V_h\label{feqn5.1} 
\end{align}
with $W^0=w_{0h}$.
To compute this, we use Freefem++ which provides an interpolation operator {\it convect} to calculate the convective term and for final plots, Matlab has been used. 
\begin{example}
In this example, we consider $w_0=\sin(\pi x_1)\times\sin(\pi x_2)-(-0.2 x_1)$ , where $u^\infty=-0.2x_1$ is the steady state solution with $\nu=0.1$ and $c_0=1$ in $\Omega=[0,1]\times [0,1]$  with time step $k=0.0005$.
We take zero Neumann boundary conditions which is without control and denoted
it as "Uncontrolled solution" in Figure \ref{fig:1.1}. For controlled solution, we choose the control \eqref{feqx1} and denoted it as "controlled solution with cubic nonlinear law" in Figure \ref{fig:1.1}.
\end{example}
\begin{figure}[ht!]
	\begin{minipage}[b]{.5\linewidth}
		\centering
		
		\includegraphics[height=6cm]{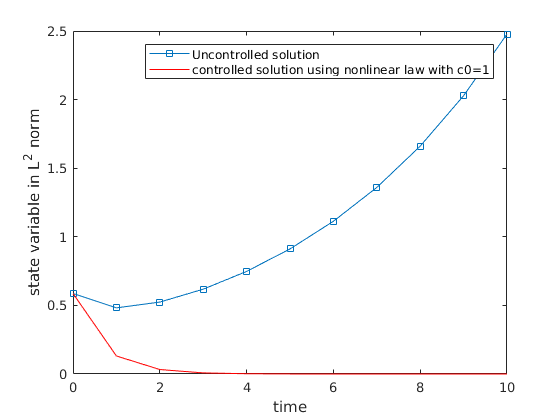}
		\caption{Both uncontrolled and controlled solution in $L^2(\Omega)$ norm}
		\label{fig:1.1}
	\end{minipage}
	\hspace{0.05cm} 
	\begin{minipage}[b]{0.5\linewidth}
		\centering
		\includegraphics[height=6cm]{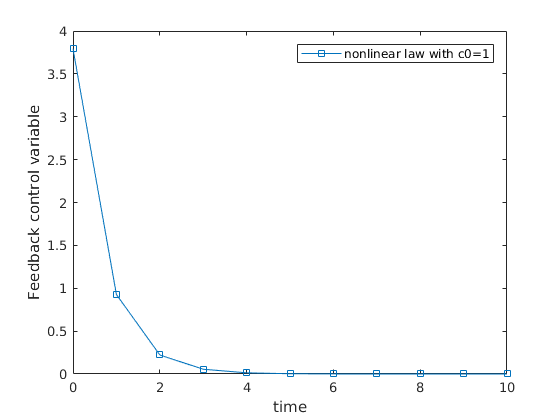}
		\caption{feedback control law in $L^2(\partial\Omega)$ norm}
		\label{fig:1.2}
	\end{minipage}
\end{figure}	
\begin{table}[ht]
	\centering
	\caption{Errors and convergence rate for the state variable $w$ with $c_0=1$ and $t=t_n=1$}
	\begin{tabular}{c c c c c}
		\hline
		$h$ & $\norm{w(t_n)-W^n}$ & Conv. Rate & $\norm{w(t_n)-W^n}_1$ & Conv. Rate
		\\  \hline \hline
		$\frac{1}{4}$ & $0.0214813$  &          & $0.153906$ & \\ \hline
		$\frac{1}{8}$ & $0.0059996$ & 1.8401 & $0.0777889$ & 0.9844 \\ \hline
		$\frac{1}{16}$ & $0.00157007$ & 1.934 & $0.0383679$ & 1.019 \\ \hline
		$\frac{1}{32}$ & $0.00041675$ & 1.9135 & $0.0185611$ & 1.047 \\ \hline
		\label{table:6.1}
	\end{tabular}
\end{table}
\begin{table}[ht]
	\centering
	\caption{Errors and convergence rate for the feedback control variable $v$ when $c_0=1$ and $t_n=1$}
	\begin{tabular}{c c c c c}
		\hline
		$h$ & $|{v(t_n)-v_{h}(t_n)}|$ & Conv. Rate 
		\\  \hline \hline
		$\frac{1}{4}$ & $0.13599$ &           \\ \hline
		$\frac{1}{8}$ & $0.0412476$ & 1.7211 \\ \hline
		$\frac{1}{16}$ & $0.0115142$ & 1.8409   \\ \hline
		$\frac{1}{32}$ & $0.00309629$ & 1.8595  \\ \hline
		\label{table:6.2}
	\end{tabular}
\end{table}
 In Figure  \ref{fig:1.1}, it is observable that with control \eqref{feqx1}, the solution for the problem \eqref{feq1.10} in $L^2$ norm goes to zero exponentially. From Table \ref{table:6.1}, it follows that $L^2$ and $H^1$ orders of convergence
 for state variable are $2$ and $1$, respectively, which confirms our theoretical results, in Theorem \ref{fthm3.3}. Take very refined mesh solution as exact solution and derive the order of convergence.
 In Table \ref{table:6.2}, it is noted that the order of convergence of nonlinear Neumann feedback control law \eqref{feqx1} is $2$, which verify our theoretical result in Theorem \ref{fthm3.3}.
 \begin{example}
 In this example, we take the initial condition $w_0=sin(\pi x_1)sin(\pi x_2)-0.2x_1-(-0.2x_1)=sin(\pi x_1)sin(\pi x_2)$ with $u^\infty=-0.2$ as the steady state solution in the domain $\Omega=[0,1]\times[0,1]$ with $\nu=0.1$ and time step $k=0.0005$. We choose the uncontrolled solution as the solution of \eqref{feq1.7} with some part on the boundary zero Dirichlet condition namely on $\Gamma_1=\{1\}\times [0,1]$ and on the remaining part ($\Gamma-\Gamma_1$) zero Neumann boundary condition and denoted it as "Uncontrolled solution " in Figure \ref{fig:1.3}. For the controlled solution,  we take the solution of \eqref{feq1.7} with feedback control law \eqref{feqx1}  on the remaining Neumann boundary part $\Gamma-\Gamma_1$ with $c_0=1$ and denoted it as "controlled solution using nonlinear law with $c0=1$"
  in Figure \ref{fig:1.3}.
\end{example}
\begin{figure}[ht!]
	\begin{minipage}[b]{.5\linewidth}
		\centering
		
		\includegraphics[height=6cm]{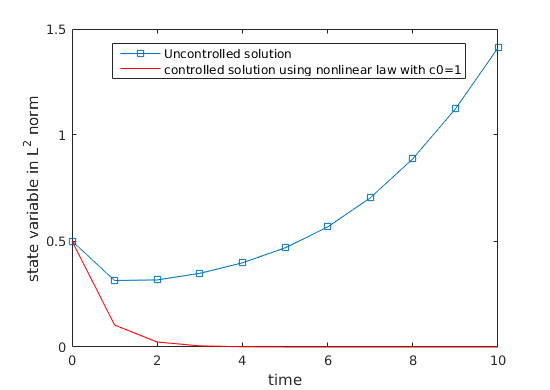}
		\caption{Both uncontrolled and controlled solution with two cases in $L^2(\Omega)$ norm}
		\label{fig:1.3}
	\end{minipage}
	\hspace{0.05cm} 
	\begin{minipage}[b]{0.5\linewidth}
		\centering
		\includegraphics[height=6cm]{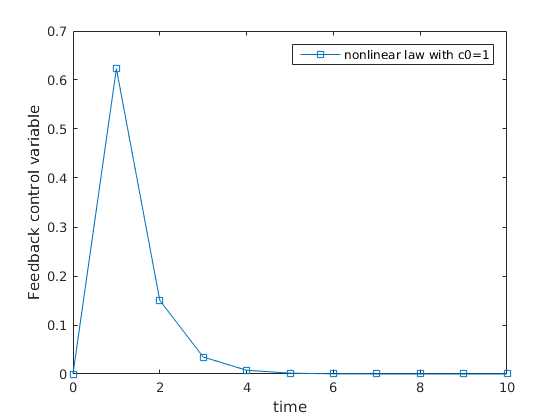}
		\caption{Feedback control law in $L^2(\partial\Omega)$ norm}
		\label{fig:1.4}
	\end{minipage}
\end{figure}	
From Figure \ref{fig:1.3}, it is documented that as time increases the uncontrolled solution does not go to zero. But with feedback control law \eqref{feqx1}, the solution of \eqref{feq1.7} goes
to zero.
Corresponding feedback control law settle down at zero as time increases which is documented in Figure \ref{fig:1.4}.
\section*{Acknowledgments}
The first author was supported by the ERC advanced grant 668998
(OCLOC) under the EU's H2020 research program. Jean-Pierre Raymond is gratefully acknowledged for his constructive remarks and suggestions which complete the paper during first author's visit to him in Universitat paul Sabatier, Toulouse.
\bibliographystyle{amsplain}

\begin{thebibliography}{50}
 \bibitem{adams2003}  R. A. Adams and  J. J. F. Fournier, \emph{Sobolev spaces}, 
 Elsevier/Academic Press, Amsterdam, 2003.
 \bibitem{agmons10}  S. Agmon, \emph{Lectures on elliptic boundary value problems}, 
AMS Chelsea Publishing, Providence, RI, 2010.
 \bibitem{balogh}  A. Balogh and  M. Krstic, \emph{Burgers' equation with nonlinear boundary feedback: $H^1$  stability well-posedness and simulation}, Math. Problems Engg. 6(2000), pp. 189--200.
 \bibitem{raymond2015}  J. M. Buchot, J. P. Raymond and  J. Tiago, \emph{Coupling estimation and control for a two dimensional
 	{B}urgers type equation}, ESAIM Control Optim. Calc. Var. 21(2015), pp. 535--560.
  
   \bibitem{bk}  J. A. Burns and  S. Kang, \emph{A control problem for Burgers' equation with bounded input/output}, Nonlinear Dynamics 2 (1991), pp. 235--262.
   \bibitem{bk1}  J. A. Burns and S. Kang,  \emph{A stabilization problem for Burgers’ equation with unbounded control and observation}, Proceedings of an International Conference on Control and Estimation of Distributed Parameter Systems, Vorau, July 8–14, 1990.
  \bibitem{bgs}  C. I. Byrnes,  D. S. Gilliam and  V. I. Shubov, \emph{On the global dynamics of a controlled viscous Burgers' equation}, J. Dynam. Control Syst. 4(1998), pp. 457--519.
  
   \bibitem{camphouse04} R. Chris Camphouse and James Myatt, \emph{Feedback Control for a Two-Dimensional Burgers' Equation System Model}, 2nd AIAA Flow Control Conference Portland, Oregon, 28 June-1 July, 2004.
   
  \bibitem{cannon99}  J. R. Cannon, R. E. Ewing,  Y. He and  Y-P. Lin, \emph{A modified nonlinear Galerkin method for the viscoelastic fluid motion equations}, International journal of Engineering Science 37(1999), pp. 1643--1662.
 
%
  \bibitem{dt73} J. Douglas and T. Dupont, \emph{Galerkin methods for parabolic equations with nonlinear boundary conditions}, Numer.Math. 20(1973), pp. 213--237.
  
  
  \bibitem{ik}  K. Ito and  S. Kang, \emph{A dissipative feedback control for systems arising in fluid dynamics}, SIAM J. Control Optim. 32(1994), pp. 831--854.
  \bibitem{iy}  K. Ito and  Y. Yan, \emph{Viscous scalar conservation laws with nonlinear flux feedback and global attractors}, J. Math. Anal. Appl. 227(1998), pp. 271--299.
  
 \bibitem{krstic1}  M. Krstic, \emph{On global stabilization of Burgers' equation by boundary control}, Systems Control Lett. 37(1999), pp. 123-141.
  \bibitem{skakp1}  S. Kundu and  A. K. Pani, \emph{Finite element approximation to global stabilization of the Burgers' equation by Neumann boundary feedback control law}, Advances in Computational Mathematics 44(2018), pp. 541--570. 
  \bibitem{skakp2}  S. Kundu and  A. K. Pani, \emph{Global stabilization of 2D-Burgers' equation by nonlinear Neumann boundary feedback control and its finite element analysis}, arXiv:1812.02083.
  \bibitem{lm68} J. L. Lions and  E. Magenes: \emph{Probl\`emes aux limites non homog\`enes et applications} Paris : Dunod 1968.
 \bibitem{liu}  W. J. Liu and  M. Krstic, \emph{Adaptive control of Burgers equation with unknown viscosity}, International Journal of Adaptive Control and Signal Process 15(2001), pp. 745-766.
 \bibitem{lmt} H. V.  Ly and  K. D. Mease and  E. S. Titi, \emph{Distributed and boundary control of the viscous Burgers' equation}, Numer. Funct. Anal. Optim. 18(1997), pp. 143--188.
 \bibitem{nirenberg59}  L. Nirenberg, \emph{On elliptic partial differential equations}, Ann. Scuola Norm. Sup. Pisa (3) 13(1959), pp. 115--162.
 \bibitem{Raymond06}  J. P. Raymond, \emph{Feedback boundary stabilization of the two-dimensional
 	{N}avier-{S}tokes equations}, SIAM J. Control Optim. 45(2006) pp. 790--828
  
 \bibitem{Smaoui}  N. Smaoui, \emph{ Nonlinear boundary control of the generalized Burgers equation}, Nonlinear Dynam. 37(2004) pp. 75--86.
 \bibitem{smaoui1} N. Smaoui, \emph{Boundary and distributed control of the viscous Burgers equation}, J. Comput. Appl. Math. 182(2005), pp. 91--104.
 \bibitem{raymond2010}  L. Thevenet,  J. M. Buchot and 
 J. P. Raymond, \emph{Nonlinear feedback stabilization of a two-dimensional
 	{B}urgers' equation}, ESAIM Control Optim. Calc. Var. 16(2010), pp. 929--955.
 
 \bibitem{thomee}  V. Thomee, \emph{Galerkin finite element methods for parabolic problems}, Springer, Berlin 1997.
 \bibitem{kesavan} Kesavan, S.: \emph{Topics in Functional Analysis and Application}, 
 New Age International (P)Ltd Publishers, New Delhi, 2008.
%
%
%
%
%
 \end{thebibliography}
 
\end{document}